\documentclass{amsart}
\usepackage[all]{xy}
\usepackage{color}
\usepackage{stmaryrd}
\usepackage{amsthm}
\usepackage{amssymb}
\usepackage[colorlinks=true]{hyperref}

\numberwithin{equation}{section}

\newtheorem{theorem}[equation]{Theorem}
\newtheorem*{theorem*}{Theorem}
\newtheorem{lemma}[equation]{Lemma}

\newtheorem*{conjecture*}{Mamma Conjecture}
\newtheorem*{conjecture1*}{Mamma Conjecture (revisited)}
\newtheorem{proposition}[equation]{Proposition}
\newtheorem{corollary}[equation]{Corollary}
\newtheorem*{corollary*}{Corollary}

\theoremstyle{remark}
\newtheorem{definition}[equation]{Definition}

\newtheorem{example}[equation]{Example}

\theoremstyle{remark}
\newtheorem{remark}[equation]{Remark}

\newcommand{\cA}{{\mathcal A}}

\newcommand{\cC}{{\mathcal C}}
\newcommand{\cD}{{\mathcal D}}
\newcommand{\cE}{{\mathcal E}}
\newcommand{\cF}{{\mathcal F}}
\newcommand{\cG}{{\mathcal G}}

\newcommand{\cM}{{\mathcal M}}
\newcommand{\cN}{{\mathcal N}}
\newcommand{\cO}{{\mathcal O}}

\newcommand{\cQ}{{\mathcal Q}}
\newcommand{\cR}{{\mathcal R}}

\newcommand{\cZ}{{\mathcal Z}}
\newcommand{\dgHo}{\mathrm{H}^0}

\newcommand{\bbC}{\mathbb{C}}

\newcommand{\bbL}{\mathbb{L}}
\newcommand{\bbN}{\mathbb{N}}
\newcommand{\bbP}{\mathbb{P}}
\newcommand{\bbR}{\mathbb{R}}

\newcommand{\bbQ}{\mathbb{Q}}
\newcommand{\bbZ}{\mathbb{Z}}

\DeclareMathOperator{\SmProj}{SmProj} 

\DeclareMathOperator{\id}{id}

\DeclareMathOperator{\Mod}{Mod}

\DeclareMathOperator{\Num}{Num} 

\newcommand{\NC}{\mathrm{nc}} 

\newcommand{\NChow}{\mathrm{NChow}} 
\newcommand{\NNum}{\mathrm{NNum}} 
\newcommand{\NVoev}{\mathrm{NVoev}} 
\newcommand{\NHom}{\mathrm{NHom}} 
\newcommand{\Voev}{\mathrm{Voev}} 

\newcommand{\dgcat}{\mathrm{Dgcat}}
\newcommand{\dgalg}{\mathrm{Dga}}
\newcommand{\spdgalg}{\mathrm{SpDga}}

\newcommand{\perf}{\mathrm{perf}}

\newcommand{\Chow}{\mathrm{Chow}}

\newcommand{\dg}{\mathrm{dg}}

\newcommand{\Hom}{\mathrm{Hom}}
\newcommand{\End}{\mathrm{End}}

\newcommand{\rep}{\mathrm{rep}}

\newcommand{\Hmo}{\mathrm{Hmo}}
\newcommand{\op}{\mathrm{op}}

\newcommand{\stab}{\mathrm{stab}}

\newcommand{\too}{\longrightarrow}

\newcommand{\ie}{\textsl{i.e.}\ }
\newcommand{\eg}{\textsl{e.g.}}

\begin{document}

\title[Weil restriction of noncommutative motives]{Weil restriction of noncommutative motives}
\author{Gon{\c c}alo~Tabuada}

\address{Gon{\c c}alo Tabuada, Department of Mathematics, MIT, Cambridge, MA 02139, USA}
\email{tabuada@math.mit.edu}
\urladdr{http://math.mit.edu/~tabuada}
\thanks{The author was partially supported by a NSF CAREER Award}

\date{\today}


\abstract{The Weil restriction functor, introduced in the late fifties, was recently extended by Karpenko to the category of Chow motives with integer coefficients. In this article we introduce the noncommutative (=NC) analogue of the Weil restriction functor, where schemes are replaced by dg algebras, and extend it to Kontsevich's categories of NC Chow motives and NC numerical motives. Instead of integer coefficients, we work more generally with coefficients in a binomial ring. Along the way, we extend Karpenko's functor to the classical category of numerical motives, and compare this extension with its NC analogue. As an application, we compute the (NC) Chow motive of the Weil restriction of every smooth projective scheme whose category of perfect complexes admits a full exceptional collection. Finally, in the case of central simple algebras, we describe explicitly the NC analogue of the Weil restriction functor using solely the degree of the field extension. This leads to a ``categorification'' of the classical corestriction homomorphism between Brauer groups.
}}

\maketitle
\vskip-\baselineskip
\vskip-\baselineskip
\vskip-\baselineskip



\section{Introduction}
\subsection*{Weil restriction}
Given a finite separable field extension $l/k$, Weil~\cite{Weil} introduced in the late fifties the {\em Weil restriction functor}
\begin{equation}\label{eq:Weil}
\cR_{l/k}: \mathrm{QProj}(l) \too \mathrm{QProj}(k) 
\end{equation}
from quasi-projective $l$-schemes to quasi-projective $k$-schemes. This functor is nowadays an important tool in algebraic geometry and number theory; see Milne's work \cite{Milne} on the Swinnerton-Dyer conjecture. Conceptually, \eqref{eq:Weil} is the right adjoint of the base-change functor. Among other properties, it preserves smoothness, projectiveness, and it is moreover symmetric monoidal. Hence, it restricts to a $\otimes$-functor 
\begin{equation}\label{eq:Weil-smooth}
\cR_{l/k}: \SmProj(l) \too \SmProj(k)
\end{equation}
from smooth projective $l$-schemes to smooth projective $k$-schemes. At the beginning of the millennium, Karpenko \cite{Karpenko} extended \eqref{eq:Weil-smooth} to $\otimes$-functors
\begin{equation}\label{eq:Karpenko}
\xymatrix@C=2.9em@R=2em{
\SmProj(l)^\op\ar[d]_-{M} \ar[r]^{\cR_{l/k}} & \SmProj(k)^\op \ar[d]^-{M} & \SmProj(l)^\op\ar[d]_-{M^\ast} \ar[r]^{\cR_{l/k}} & \SmProj(k)^\op \ar[d]^-{M^\ast} \\
\Chow(l) \ar[r]_-{\cR_{l/k}} & \Chow(k) & \Chow^\ast(l) \ar[r]_-{\cR^\ast_{l/k}} & \Chow^\ast(k)
}
\end{equation}
defined on the categories of Chow motives with integer coefficients; $\Chow^\ast(-)$ is constructed using correspondences of arbitrary codimension. These latter $\otimes$-functors, although well-defined, are not additive! Consequently, they are not the right adjoints of the corresponding base-change functors.
\subsection*{Noncommutative motives}
In noncommutative algebraic geometry in the sense of Bondal, Drinfeld, Kaledin, Kapranov, Kontsevich, Orlov, Stafford, Van den Bergh, and others (see \cite{BK,BO1,BO,BV,VDrinfeld,Kaledin,IAS,Miami,finMot,Stafford, Stafford1}), schemes are replaced by differential graded (=dg) algebras. A celebrated result, due to Bondal-Van den Bergh \cite{BV}, asserts that for every quasi-compact quasi-separated scheme $X$ there exists a dg algebra $A_X$ whose derived category $\cD(A_X)$ is equivalent to $\cD(X)$; see \S\ref{sec:schemes}. The dg algebra $A_X$ is unique up to Morita equivalence and reflects many of the properties of $X$. For example, $X$ is smooth (resp. proper) if and only if $A_X$ is smooth (resp. proper) in the sense of Kontsevich; see \S\ref{sec:dg}. Let $\dgalg(k)$ be the category of dg $k$-algebras and $\spdgalg(k)$ its full subcategory of smooth proper dg $k$-algebras. Making use of $K$-theory, Kontsevich introduced in \cite{IAS,Miami,finMot} the category of noncommutative Chow motives with integer coefficients $\NChow(k)$ as well as a canonical $\otimes$-functor $
U: \spdgalg(k) \to \NChow(k)$; consult \S\ref{sec:NCmotives} for details and the survey articles \cite{survey,survey1} for several applications.
\subsection*{Motivation:}
The aforementioned constructions and results in the commutative world lead us naturally to the following motivating questions:
\begin{itemize}
\item[{\bf Q1:}] {\it Does the Weil restriction functor admits a noncommutative analogue $\cR^\NC_{l/k}$ ?}
\item[{\bf Q2:}] {\it Does $\cR_{l/k}^\NC$ extends to a $\otimes$-functor on noncommutative Chow motives ?}
\item[{\bf Q3:}] {\it What is the relation between the functors $\cR_{l/k}$ and $\cR^\NC_{l/k}$ ?}
\end{itemize}
Besides Chow motives, we can consider Grothendieck's category of numerical motives. The noncommutative analogue of this category was also introduced by Kontsevich in \cite{IAS}; see \S\ref{sec:NCmotives}. Hence, it is natural to ask the following extra question:
\begin{itemize}
\item[{\bf Q4:}] {\it Do the functors $\cR_{l/k}$ and $\cR^\NC_{l/k}$ extend to numerical motives ?}
\end{itemize}
In this article we provide precise answers to all these questions. In the process we develop new technology of independent interest. Consult \S\ref{sec:computations} for computations.
\section{Statement of results}\label{sec:statements}
Let $l/k$ be a finite separable field extension of degree $d$, $L/k$ a finite Galois field extension containing $l$, $G$ the Galois group $\mathrm{Gal}(L/k)$, $H$ the subgroup of $G$ such that $L^H\simeq l$, and $G/H$ the $G$-set of left cosets of $H$ in $G$. Under these notations, the noncommutative analogue of \eqref{eq:Weil} is given by the following $\otimes$-functor
\begin{eqnarray}\label{eq:Weil-2}
\cR^\NC_{l/k}: \dgalg(l) \too \dgalg(k) && A \mapsto (\otimes_{\sigma\in G/H} {}^{\sigma\!} A_L)^G\,,
\end{eqnarray}
where $A_L:=A\otimes_l L$, ${}^{\sigma\!} A_L$ is the $\sigma$-conjugate of $A_L$, and $G$ acts on $\otimes_{\sigma\in G/H} {}^{\sigma\!} A_L$ by permutation of the $\otimes$-factors; see \S \ref{sub:WeilNC}. Note that the $\sigma$-conjugate ${}^{\sigma\!} A_L$ depends only (up to isomorphism) on the left coset $\sigma H$ and that the functor \eqref{eq:Weil-2} is independent (up to isomorphism) of the Galois field extension $L/k$. In the particular case where $l/k$ is Galois, we can take $L=l$. Our first main result is the following:
\begin{theorem}\label{thm:main1}
The functor \eqref{eq:Weil-2} preserves smooth and proper dg algebras.
\end{theorem}
The above functor \eqref{eq:Weil-2}, combined with Theorem \ref{thm:main1}, provides an answer to question Q1. 
Our affirmative answer to question Q2 is the following: 
\begin{theorem}\label{thm:new}
The restriction of the above functor \eqref{eq:Weil-2} to smooth proper dg algebras extends to a $\otimes$-functor on noncommutative Chow motives
\begin{equation}\label{eq:we}
\cR^\NC_{l/k}: \NChow(l) \too \NChow(k) \,.
\end{equation}
\end{theorem}
As in the commutative world, the functor \eqref{eq:we} is not additive! Therefore, given a commutative ring of coefficients $R$, this functor (as well as Karpenko's functors \eqref{eq:Karpenko}) cannot be extended by $R$-linearity to noncommutative Chow motives with $R$-coefficients. When $R$ is a binomial ring (see \S\ref{sec:binomial}), this difficulty can be~circumvented:
\begin{proposition}\label{prop:binomial}
Given a binomial ring $R$, \eqref{eq:Karpenko} and \eqref{eq:we} extend to $\otimes$-functors
$$\Chow_R(l)\stackrel{\cR_{l/k}}{\too} \Chow_R(k) \,\,\,\, \Chow_R^\ast(l) \stackrel{\cR^\ast_{l/k}}{\too} \Chow_R^\ast(k) 
\,\,\,\, \NChow_R(l)\stackrel{\cR_{l/k}^\NC}{\too} \NChow_R(k)\,.$$
\end{proposition}
As explained in \cite[Thm.~1.1]{CvsNC}, when $R$ is an $\bbQ$-algebra there exists an $R$-linear fully-faithful $\otimes$-functor\footnote{The functor $\Psi$ was denoted by $R$ in {\em loc. cit.}} $\Psi$ such that the following composition
$$ \SmProj(k)^\op \stackrel{M_R^\ast}{\too} \Chow_R^\ast(k) \stackrel{\Psi}{\too} \NChow_R(k)$$
sends $X$ to $U_R(A_X)$. Intuitively speaking, the commutative world can be embedded into the noncommutative world. Our answer to question Q3, which also justifies the correctness of the above functor \eqref{eq:Weil-2}, is the following:
\begin{theorem}\label{thm:new1}
Given a quasi-projective $l$-scheme $X$, the dg $k$-algebras $\cR^\NC_{l/k}(A_X)$ and $A_{\cR_{l/k}(X)}$ are Morita equivalent. Moreover, we have the commutative diagram:
\begin{equation}\label{eq:diagram-Chow}
\xymatrix{
\Chow_R^\ast(l) \ar[rr]^-{\Psi} \ar[d]_-{\cR^\ast_{l/k}}&& \NChow_R(l) \ar[d]^-{\cR^{\NC}_{l/k}} \\
\Chow^\ast_R(k) \ar[rr]_-{\Psi} && \NChow_R(k)\,.
}
\end{equation}
\end{theorem}
Roughly speaking, Theorem \ref{thm:new1} shows that the bridge between the commutative and the noncommutative world is compatible with Weil restriction. Since $\Psi$ is fully-faithful, the functor $\cR_{l/k}^\NC$ should then be considered as a generalization of~$\cR^\ast_{l/k}$.

Let us denote by $\Num_R(k)$ and $\Num^\ast_R(k)$ the Grothendieck's categories of numerical motives (see Manin \cite{Manin}), and by $\NNum_R(k)$ the category of noncommutative numerical motives (see \S\ref{sec:NCmotives}). As explained in \cite[Thm.~1.13]{AJM}, the above functor $\Psi$  descends to an $R$-linear fully-faithful $\otimes$-functor $\Psi_{\mathrm{num}}: \Num_R^\ast(k) \to \NNum_R(k)$. Under these notations, our affirmative answer to question Q4 is the following:
\begin{theorem}\label{thm:new22}
The functors of Proposition \ref{prop:binomial} descend to $\otimes$-functors:
$$
\Num_R(l)\stackrel{\cR_{l/k}}{\too} \Num_R(k) \,\,\,\, \Num_R^\ast(l) \stackrel{\cR^\ast_{l/k}}{\too} \Num_R^\ast(k) \,\,\,\, \NNum_R(l) \stackrel{\cR_{l/k}^\NC}{\too} \NNum_R(k)\,.
$$
Moreover, we have the following commutative diagram:
\begin{equation}\label{eq:diagram-num}
\xymatrix{
\Num_R^\ast(l) \ar[rr]^-{\Psi_{\mathrm{num}}} \ar[d]_-{\cR^\ast_{l/k}}&& \NNum_R(l) \ar[d]^-{\cR^{\NC}_{l/k}} \\
\Num_R^\ast(k) \ar[rr]_-{\Psi_\mathrm{num}} && \NNum_R(k)\,.
}
\end{equation}
\end{theorem}
\begin{remark}[Other equivalence relations]
Theorem \ref{thm:new22} is proved not only for numerical motives but also for $\otimes$-nilpotence motives and homological motives; see~\S\ref{sec:thm22}.
\end{remark}
\section{Computations}\label{sec:computations}
Given an integer $n \geq 2$, consider the following $G$-action
\begin{eqnarray*}
 G \times \{G/H \to \{1, \ldots, n\}\} \to \{G/H \to \{1,\ldots, n\}\} && (\rho, \alpha) \mapsto  (\sigma \mapsto \alpha(\rho^{-1}\sigma))
\end{eqnarray*}
on the set of maps from $G/H$ to $\{1,\ldots, n\}$. Let us denote by $\cO(G/H,n)$ the associated set of orbits and by $\stab(\alpha) \subseteq G$ the stabilizer of $\alpha$. As explained by Karpenko in \cite[pages~83-84]{Karpenko}, we have an isomorphism\footnote{Note that $L^{\mathrm{stab}(\alpha)}$ depends only (up to isomorphism) on the orbit $\alpha \in \cO(G/H,n)$.}
\begin{equation}\label{eq:computation}
\cR_{l/k}(\underbrace{\mathrm{Spec}(l) \times \cdots \times \mathrm{Spec}(l)}_{n\text{-}\mathrm{copies}})\simeq \prod_{\alpha \in \cO(G/H,n)} \mathrm{Spec}(L^{\mathrm{stab}(\alpha)})\,. 
\end{equation}
\begin{remark}\label{rk:intermediate}
Given an intermediate field extension $l/l'/k$, let $H'$ be the subgroup of $G$ such that $L^{H'}\simeq l'$. Consider the following map
\begin{eqnarray*}
\alpha':G/H \to \{1, \ldots, n\} && \sigma H \mapsto \left\{
  \begin{array}{lcr}1  & \mathrm{if} & \sigma \in H'\\
    n & \mathrm{if} & \sigma \notin H' \\
  \end{array}
\right.\,.
\end{eqnarray*}
Since the stabilizer $\mathrm{stab}(\alpha')$ of $\alpha'$ is isomorphic to $H'$, we observe that the affine $k$-scheme $\mathrm{Spec}(l')$ appears on the right-hand side of \eqref{eq:computation}. This illustrates the highly non-additive behavior of the Weil restriction functor.
\end{remark}
Let $R$ be a binomial ring. By combining \eqref{eq:computation} with the first claim of the above Theorem \ref{thm:new1}, we hence obtain the following motivic computation:
\begin{equation}\label{eq:motivic-computation}
\cR_{l/k}^\NC(\underbrace{U_R(l) \oplus \cdots \oplus U_R(l)}_{n\text{-}\mathrm{copies}})\simeq \bigoplus_{\alpha \in \cO(G/H,n)} U_R(L^{\mathrm{stab}(\alpha)})\,.
\end{equation}
\begin{example}[Finite dimensional algebras]
Let $A$ be a finite dimensional $\bbC$-algebra of finite global dimension. Examples include path algebras of finite quivers without oriented cycles and more generally their quotients by admissible ideals (\eg\ the quiver algebras of Khovanov-Seidel \cite{Khovanov-Seidel} and the close relatives of Rouquier-Zimmerman \cite{RZ}). As explained in \cite[Rk.~3.19]{TV}, $U_R(A)$ is isomorphic to the direct sum of $n$ copies of $U_R(\bbC)$ where $n$ stands for the number of simple (right) $A$-modules. Making use of \eqref{eq:motivic-computation}, we hence obtain the following computation:
$$ \cR^\NC_{\bbC/\bbR} (U_R(A)) \simeq \underbrace{U_R(\bbR) \oplus \cdots \oplus U_R(\bbR)}_{n\text{-}\mathrm{copies}} \oplus \underbrace{U_R(\bbC) \oplus \cdots \oplus U_R(\bbC)}_{\binom{n}{2}\text{-}\mathrm{copies}}\,.$$
\end{example}
\subsection*{Full exceptional collections}
Let $X$ be a smooth projective $l$-scheme whose category of perfect complexes $\perf(X)$ admits a full exceptional collection of length $n$; see Bondal-Orlov \cite[Def.~2.3]{BO1}. As explained in \cite[Lem.~5.1]{MT}, the noncommutative Chow motive $U(A_X)$ becomes then isomorphic to the direct sum of $n$ copies of $U(l)$.
\begin{corollary}\label{cor:main}
Let $R$ be a binomial ring and $X$ an $l$-scheme as above.
\begin{itemize}
\item[(i)] We have an isomorphism $U_R(A_{\cR_{l/k}(X)})\simeq \bigoplus_{\alpha} U_R(L^{\stab(\alpha)})$.
\item[(ii)] When $R$ is a $\bbQ$-algebra, we have $M^\ast_R(\cR^\ast_{l/k}(X))\simeq \bigoplus_\alpha M^\ast_R(\mathrm{Spec}(L^{\stab(\alpha)}))$.
\item[(iii)] When $R$ is a $\bbQ$-algebra, the Chow motive $M_R(\cR_{l/k}(X))$ is a direct summand of $\bigoplus_{i=0}^{d\mathrm{dim}(X)} \bigoplus_\alpha M_R(\mathrm{Spec}(L^{\stab(\alpha)})) \otimes \bbL^{\otimes i}$ where $\bbL$ stands for the Lefschtez motive. In particular, $M_R(\cR_{l/k}(X))$ is an Artin-Tate motive.
\end{itemize}
\end{corollary}
\begin{remark}
By construction, the Weil restriction $\cR_{l/k}(X)$ is of dimension $d \mathrm{dim}(X)$. Therefore, as explained in \cite[\S2]{Marcello}, the above items (ii)-(iii) hold more generally for every binomial ring $R$ such that $\frac{1}{(2 d\mathrm{dim}(X))!} \in R$.
\end{remark}
\begin{proof}
Item (i) follows from the combination of isomorphism \eqref{eq:motivic-computation} with the first claim of Theorem \ref{thm:new1}. Item (ii) follows from the combination of item (i) with the above commutative diagram \eqref{eq:diagram-Chow}. In what concerns item (iii), it is obtained using the different (codimension) components of the isomorphism of item (ii).
\end{proof}
Thanks to Beilinson, Kapranov, Kawamata, Kuznetsov, Manin-Smirnov, Orlov, and others (see \cite{Beilinson,kapranovquadric,Kawamata,kuznetfanothreefolds,MS,orlovprojbund}), Corollary \ref{cor:main} applies to projective spaces, rational surfaces, and moduli spaces of pointed stable curves of genus zero (in the case of an arbitrary base field $l$), and to smooth quadric hypersurfaces, Grassmannians, flag varieties, Fano threefolds with vanishing odd cohomology, and toric varieties (in the case where $l$ is algebraically closed and of characteristic zero). Conjecturally, it applies also to all the homogeneous spaces of the form $G/P$, with $P$ a parabolic subgroup of a semisimple algebraic group $G$; see Kuznetsov-Polishchuk \cite{KP}.
\begin{example}[Moduli spaces]
Let $X$ be the $\bbQ(\zeta_3)$-scheme $\overline{\cM}_{0,5}$, \ie the moduli space of $5$-pointed stable curves of genus zero. As proved by Manin-Smirnov in \cite[\S3.3]{MS}, $\perf(\overline{\cM}_{0,5})$ admits a full exceptional collection of length $7$. Making use of Corollary \ref{cor:main}, we hence obtain the following computations: 
$$U_R(A_{\cR_{\bbQ(\zeta_3)/\bbQ}({\overline{\cM}_{0,5}})})\simeq U_R(\bbQ)^{\oplus 7}\oplus U_R(\bbQ(\zeta_3))^{\oplus 21}\,.$$
$$M^\ast_R(\cR^\ast_{\bbQ(\zeta_3)/\bbQ}({\overline{\cM}_{0,5}}))\simeq M^\ast_R(\mathrm{Spec}(\bbQ))^{\oplus 7}\oplus M^\ast_R(\mathrm{Spec}(\bbQ(\zeta_3)))^{\oplus 21}\,.$$
\end{example}
\begin{remark}[Incompatibility]
By combining Corollary \ref{cor:main} with Remark \ref{rk:intermediate}, we observe that if $X$ is an $l$-scheme such that $\perf(X)$ admits a full exceptional collection, then $\perf(\cR_{l/k}(X))$ does {\em not} admits a full exceptional collection! Roughly speaking, Weil restriction is always incompatible with full exceptional collections.
\end{remark}
\begin{remark}[Generalizations]
Corollary \ref{cor:main} holds more generally for every smooth projective $l$-scheme $X$ such that $U_R(A_X)\simeq \bigoplus U_R(l)$. Other examples include quadrics (see \cite[\S3]{TV}), complex surfaces of general type (see Gorchinskiy-Orlov \cite[Props.~2.3 and 4.2]{GO}), and Severi-Brauer varieties (see \cite[\S3]{TV}).
\end{remark}
\subsection*{Central simple algebras}
Let us denote by $\mathrm{CSA}(k)$ the full symmetric monoidal subcategory of $\NChow(k)$ consisting of the objects $U(A)$ with $A$ a central simple $k$-algebra. Following \cite[Prop.~2.25]{separable}, we have natural isomorphisms ($\mathrm{ind}$=index)
\begin{equation}\label{eq:index}
\Hom_{\mathrm{CSA}(k)}(U(A),U(A'))\simeq \mathrm{ind}(A^\op \otimes A')\cdot \bbZ\,,
\end{equation}
under which the composition law of $\mathrm{CSA}(k)$ corresponds to multiplication. As explained by Riehm in \cite[\S5.4]{Riehm}, the assignment $A \mapsto (\otimes_{\sigma\in G/H} {}^{\sigma\!} A_L)^G$ preserves central simple algebras. Consequently, \eqref{eq:we} restricts to a $\otimes$-functor
\begin{equation}\label{eq:functor-CSA}
\cR^\NC_{l/k}:\mathrm{CSA}(l) \too \mathrm{CSA}(k)\,.
\end{equation}
\begin{theorem}\label{thm:new3}
Given central simple $l$-algebras $A$ and $A'$, the associated map 
$$ \Hom_{\mathrm{CSA}(l)}(U(A),U(B)) \too \Hom_{\mathrm{CSA}(k)}(U(\cR^\NC_{l/k}(A)),U(\cR^\NC_{l/k}(A')))$$
identifies, under the above isomorphisms \eqref{eq:index}, with the polynomial map  
\begin{eqnarray*}
\mathrm{ind}(A^\op \otimes  A')\cdot \bbZ \too \mathrm{ind}(\cR^\NC_{l/k}(A)^\op \otimes \cR^\NC_{l/k}(A'))\cdot \bbZ &&  n \mapsto n^d\,.
\end{eqnarray*}
\end{theorem}
Roughly speaking, Theorem \ref{thm:new3} shows that in the case of central simple algebras the highly non-additive behavior of the Weil restriction functor is completely determined by the degree of the field extension $l/k$. In Proposition \ref{prop:base2}, we describe also the (additive) behavior of the base-change functor $-\otimes_k l: \mathrm{CSA}(k) \to \mathrm{CSA}(l)$.
\begin{remark}[Corestriction]
Let $\mathrm{Br}(k)$ be the Brauer group of $k$. Recall from Riehm \cite[Thm.~11]{Riehm} that the corestriction homomorphism $\mathrm{cor}_{l/k}: \mathrm{Br}(l) \to \mathrm{Br}(k)$ is induced by the assignment $A \mapsto (\otimes_{\sigma\in G/H} {}^{\sigma\!} A_L)^G$. As proved in \cite[Thm.~9.1]{Twisted}, $U(A) \simeq U(A')$ if and only if $[A]=[A'] \in \mathrm{Br}(k)$. Consequently, we observe that the above $\otimes$-functor \eqref{eq:functor-CSA} ``categorifies'' the classical homomorphism $\mathrm{cor}_{l/k}$.
\end{remark}
\subsection*{Notations} Throughout the article, $l/k$ denotes a finite separable field extension of degree $d$, $L/k$ a finite Galois field extension containing $l$, $G$ the Galois group $\mathrm{Gal}(L/k)$, $H$ the subgroup of $G$ such that $L^H\simeq l$, and $G/H$ the $G$-set of left cosets of $H$ in $G$.  We will reserve the letter $R$ for a (binomial) ring of coefficients. Finally, the base-change functor from $l$ to $L$ will be denoted by $(-)_L$.
\subsection*{Acknowledgments} The author is very grateful to Joseph Ayoub, Antoine Touz{\'e}, and Michel Van den Bergh for discussions concerning equivalence relations on algebraic cycles, binomial rings, and Galois descent, respectively. He is also grateful to Roman Bezrukavnikov, Thomas Bitoun, Eric M. Friedlander, Henri Gillet, Dmitry Kaledin, Mikhail Kapranov, Nikita Karpenko, Max Karoubi, Andrei Suslin, and Marius Wodzicki, for useful conversations. The author would like also to thank the anonymous referee for his/her comments and suggestions which greatly allowed the improvement of the article. This work was initiated at the program ``Noncommutative Algebraic Geometry and Representation Theory'', MSRI, Berkeley,~2013.
\section{Differential graded preliminaries}\label{sec:dg}
Let $\cC(k)$ be the category of cochain complexes of $k$-vector spaces. A {\em differential graded (=dg) category $\cA$} is a category enriched over $\cC(k)$ and a {\em dg functor} $F:\cA\to \cA'$ is  a functor enriched over $\cC(k)$; consult Keller's ICM survey \cite{ICM}. Note that a dg algebra $A$ is a dg category with a single object. In what follows we write $\dgcat(k)$ for the category of (small) dg categories and $\dgalg(k)$ for the category of dg algebras.

Let $\cA$ be a dg category.  The category $\dgHo(\cA)$ has the same objects as $\cA$ and
morphisms $\dgHo(\cA)(x,y):=H^0\cA(x,y)$, where $H^0$
stands for degree zero cohomology. The opposite dg category $\cA^\op$ has the same objects as $\cA$ and complexes of morphisms $\cA^\op(x,y):=\cA(y,x)$.  A {\em right $\cA$-module} is a dg functor $M:\cA^\op \to \cC_\dg(k)$ with values in the dg category $\cC_\dg(k)$ of cochain complexes of $k$-vector spaces. Let us denote by $\cC(\cA)$ the category of right $\cA$-modules. The {\em derived category $\cD(\cA)$ of $\cA$} is defined as the localization of $\cC(\cA)$ with respect to the class of quasi-isomorphisms. Its full triangulated subcategory of compact objects will be denoted by $\cD_c(\cA)$.

The {\em tensor product $\cA\otimes\cA'$} of dg categories is defined as follows: the set of objects is the cartesian product and $(\cA\otimes\cA')((x,x'),(y,y')):= \cA(x,y) \otimes \cA'(x',y')$. This gives rise to a symmetric monoidal structure on $\dgcat(k)$ with $\otimes$-unit $k$. Given dg categories $\cA$ and $\cA'$, a {\em $\cA\text{-}\cA'$-bimodule} is a dg functor $\mathrm{B}: \cA \otimes \cA'^\op \to \cC_\dg(k)$. Standard examples are given by the $\cA\text{-}\cA$-bimodule
\begin{eqnarray}\label{eq:bimodule1}
\cA\otimes\cA^\op \to \cC_\dg(k) && (x,y) \mapsto \cA(y,x)
\end{eqnarray} 
and more generally by the $\cA\text{-}\cA'$-bimodule
\begin{eqnarray}\label{eq:bimodule2}
{}_F\mathrm{B}:\cA\otimes \cA'^\op \to \cC_\dg(k) && (x,x') \mapsto \cA'(x',F(x))
\end{eqnarray}
associated to a dg functor $F:\cA\to \cA'$. 

A dg functor $F:\cA\to \cA'$ is called a {\em Morita equivalence} if the restriction functor $\cD(\cA') \to \cD(\cA)$ is an equivalence of (triangulated) categories; see \cite[\S4.6]{ICM}. As proved in \cite[Thm.~5.3]{IMRN}, $\dgcat(k)$ carries a Quillen model structure whose weak equivalences are the Morita equivalences. Let us denote by $\Hmo(k)$ the associated homotopy category obtained. As explained in \cite[Cor.~5.10]{IMRN}, we have a bijection 
\begin{equation}\label{eq:bij}
\Hom_{\Hmo(k)}(\cA,\cA')\simeq \mathrm{Iso}\,\rep(\cA,\cA')\,,
\end{equation}
where the right-hand side stands for the set of isomorphism class of the full triangulated subcategory 
$\rep(\cA,\cA') \subset \cD(\cA^\op \otimes \cA')$ consisting of those
$\cA\text{-}\cA'$-bimodules $\mathrm{B}$ such that for every object $x \in\cA$ the right $\cA'$-module $\mathrm{B}(x,-)$ belongs to
$\cD_c(\cA')$. Under the above bijection \eqref{eq:bij}, the composition law of $\Hmo(k)$
corresponds to the (derived) tensor product of bimodules. Moreover, the
identity of an object $\cA$ is given by the class of the above $\cA\text{-}\cA$-bimodule \eqref{eq:bimodule1}.
\subsection*{Smoothness and properness}
Recall from Kontsevich \cite{IAS,Miami,finMot} that a dg category $\cA$ is called {\em smooth} if the $\cA\text{-}\cA$-bimodule \eqref{eq:bimodule1} belongs to $\cD_c(\cA^\op\otimes \cA)$ and {\em proper} if $\sum_i \mathrm{dim}\,H^i\cA(x,y)< \infty$ for any ordered pair of objects $(x,y)$. In what follows, we write $\spdgalg(k)$ for the category of smooth proper dg $k$-algebras. Thanks to \cite[Prop.~5.10]{CT1} and \cite[Thm.~4.12]{ICM}, every smooth proper dg category is Morita equivalent to a smooth proper dg algebra.
\section{Noncommutative motives}\label{sec:NCmotives}
Recall from \S\ref{sec:dg} the description of the homotopy category $\Hmo(k)$. Since the above
$\cA\text{-}\cA'$-bimodule \eqref{eq:bimodule2} belongs to
$\rep(\cA,\cA')$, we have the following functor
\begin{eqnarray}\label{eq:functor1}
\dgcat(k) \too \Hmo(k) &\cA \mapsto \cA & F \mapsto {}_F\mathrm{B}\,.
\end{eqnarray}
The tensor product of dg categories descends to $\Hmo(k)$ giving rise to a symmetric monoidal structure and making the above functor \eqref{eq:functor1} symmetric monoidal. The {\em additivization} of $\Hmo(k)$ is the additive category $\Hmo_0(k)$ with the same objects as $\Hmo(k)$ and morphisms $\Hom_{\Hmo_0(k)}(\cA,\cA'):=K_0\rep(\cA,\cA')$, where $K_0\rep(\cA,\cA')$ stands for the Grothendieck group of the triangulated category $\rep(\cA,\cA')$. The composition law is induced by the (derived) tensor product of bimodules. Note that we have a canonical functor
\begin{eqnarray}\label{eq:functor2}
\Hmo(k) \too \Hmo_0(k) &\cA \mapsto \cA& \mathrm{B} \mapsto [\mathrm{B}]\,.
\end{eqnarray}
Given a commutative ring of coefficients $R$, the {\em $R$-linearization} of $\Hmo_0(k)$ is the $R$-linear additive category $\Hmo_0(k)_R$ obtained by tensoring each abelian group of morphisms of $\Hmo_0(k)$ with $R$. We have an associated functor
\begin{eqnarray}\label{eq:functor3}
\Hmo_0(k) \too \Hmo_0(k)_R &\cA \mapsto \cA& [\mathrm{B}] \mapsto [\mathrm{B}]_R:=[\mathrm{B}]\otimes_\bbZ R\,.
\end{eqnarray}
The symmetric monoidal structure on $\Hmo(k)$ descends first to a bilinear symmetric monoidal structure on $\Hmo_0(k)$ and then to a $R$-linear bilinear symmetric monoidal structure on $\Hmo_0(k)_R$, making the above functors \eqref{eq:functor2}-\eqref{eq:functor3} symmetric monoidal. We hence obtain the following composition of $\otimes$-functors:
\begin{equation}\label{eq:composition-new}
\dgcat(k) \stackrel{\eqref{eq:functor1}}{\too} \Hmo(k) \stackrel{\eqref{eq:functor2}}{\too} \Hmo_0(k)\stackrel{\eqref{eq:functor3}}{\too} \Hmo_0(k)_R\,.
\end{equation}
\subsection*{Noncommutative Chow motives}
The category of {\em noncommutative Chow motives} with $R$-coefficients $\NChow_R(k)$ is defined as the idempotent completion of the full subcategory of $\Hmo_0(k)_R$ consisting of the smooth proper dg algebras. As explained in \cite[Thm.~5.8]{CT1}, $\NChow_R(k)$ is a rigid symmetric monoidal category. The restriction of \eqref{eq:composition-new} to smooth proper dg algebras gives rise to a $\otimes$-functor 
$$U_R: \spdgalg(k) \too \NChow_R(k)\,.$$
When $R=\bbZ$, we will write $\NChow(k)$ instead of $\NChow_\bbZ(k)$ and $U$ instead of $U_\bbZ$.

Given smooth proper dg algebras $A, A'$, the triangulated category $\rep(A,A')$ identifies with $\cD_c(A^\op \otimes A')$; see \cite[\S5]{CT1}. Consequently, we have the isomorphisms
\begin{equation}\label{eq:Hom-smooth}
 \Hom_{\NChow_R(k)}(U_R(A), U_R(A')) \simeq K_0\cD_c(A^\op \otimes A')=: K_0(A^\op\otimes A')_R\,.
 \end{equation}
\subsection*{Noncommutative $\otimes$-nilpotent motives}
Assume that $\bbQ \subseteq R$. Given an $R$-linear additive rigid monoidal category $(\cC,\otimes, {\bf 1})$, its $\otimes_{\mathrm{nil}}$-ideal is defined~as:
$$ \otimes_{\mathrm{nil}}(a,b):= \{ f \in \Hom_\cC(a,b)\,|\, f^{\otimes n}=0\,\,\mathrm{for}\,\,n \gg 0\}\,.$$
The category of {\em noncommutative $\otimes$-nilpotent motives} with $R$-coefficients $\NVoev_R(k)$ is defined as the idempotent completion of the quotient category $\NChow_R(k)/\otimes_{\mathrm{nil}}$.
\subsection*{Noncommutative homological motives}
Assume that $R$ is a field. As explained in \cite[Thm.~7.2]{Galois}, when $k/R$ (resp. $R/k$) periodic cyclic homology gives rise to an $R$-linear $\otimes$-functor with values in the category of $\bbZ/2$-graded vector spaces
\begin{eqnarray}\label{eq:HP}
& \NChow_R(k) \stackrel{HP^\pm}{\too} \mathrm{Vect}_{\bbZ/2}(k) & (\mathrm{resp.} \,\,\NChow_R(k) \stackrel{HP^\pm}{\too} \mathrm{Vect}_{\bbZ/2}(R))\,.
\end{eqnarray} 
Let us denote by $\mathrm{Ker}(HP^\pm)$ the associated $\otimes$-ideal. The category of {\em noncommutative homological motives} with $R$-coefficients $\NHom_R(k)$ is defined as the idempotent completion of the quotient category $\NChow_R(k)/\mathrm{Ker}(HP^\pm)$. 
\subsection*{Noncommutative numerical motives}
Given an $R$-linear additive rigid symmetric monodical category $(\cC,\otimes, {\bf 1})$, its $\otimes$-ideal $\cN$ is defined as 
$$ \cN(a,b):=\{f \in \Hom_\cC(a,b)\,\,|\,\,\forall g \in \Hom_\cC(b,a)\,\,\mathrm{we}\,\,\mathrm{have}\,\,\mathrm{trace}(g\circ f)=0\}\,,$$
where $\mathrm{trace}(-)$ stands for the categorical trace. The category of {\em noncommutative numerical motives} with $R$-coefficients $\NNum_R(k)$ is defined as the idempotent completion of the quotient category $\NChow_R(k)/\cN$.
\section{DG Galois descent}\label{sec:Galois}
In this section we extend the classical Galois descent theory to the differential graded setting; see Proposition~\ref{prop:Galois-descent}. Making use of it, we then introduce the noncommutative analogue of the Weil restriction functor; see \S\ref{sub:WeilNC}. 

For every $\sigma \in G$ we have the following $\otimes$-equivalence of categories
\begin{eqnarray*}
{}^{\sigma\!}(-): \cC(L) \stackrel{\simeq}{\too} \cC(L) && V \mapsto {}^{\sigma\!} V\,,
\end{eqnarray*}
where ${}^{\sigma\!} V$ is obtained from $V$ by restriction of scalars along $\sigma^{-1}:L\stackrel{\sim}{\to} L$. 
\begin{definition}\label{def:Galois}
A {\em $L/k$-Galois complex} is a complex of $L$-vector spaces $W$ equipped with a left $G$-action $ G \times W \to W, (\rho,w) \mapsto \rho(w)$, which is {\em skew-linear} in the sense that $\rho(\lambda)\cdot\rho(w)=\rho(\lambda\cdot w)$ for every $\lambda \in L$, $w \in W$, and $\rho \in G$.
\end{definition}
\begin{example}
Given a complex of $k$-vector spaces $V$, the complex of $L$-vector spaces $V \otimes_k L$, equipped with the skew-linear left $G$-action
\begin{eqnarray*}
G \times (V \otimes_kL) \too V\otimes_k L && (\rho,v \otimes \lambda) \mapsto v \otimes \rho(\lambda)\,,
\end{eqnarray*}
is a $L/k$-Galois complex.
\end{example} 
\begin{example}\label{ex:2}
Given a complex of $k$-vector spaces $V$, the complex of $L$-vector spaces $\otimes_{\sigma \in G} {}^{\sigma\!} V$, equipped with the skew-linear left $G$-action
\begin{eqnarray*}
G \times \otimes_{\sigma \in G} {}^{\sigma\!} V \too \otimes_{\sigma \in G} {}^{\sigma\!} V&& (\rho, \otimes_{\sigma \in G}v_\sigma) \mapsto \otimes_{\sigma \in G} v_{\rho^{-1}\sigma}\,,
\end{eqnarray*}
is a $L/k$-Galois complex. 
\end{example}
\begin{example}\label{ex:skew}
Given a complex of $l$-vector spaces $V$, the complex of $L$-vector spaces $\otimes_{\sigma \in G/H} {}^{\sigma\!} V_L$, equipped with the skew-linear left $G$-action
\begin{eqnarray*}
G \times \otimes_{\sigma \in G/H} {}^{\sigma\!} V_L \too \otimes_{\sigma \in G/H} {}^{\sigma\!} V_L && (\rho, \otimes_{\sigma \in G/H}v_\sigma) \mapsto \otimes_{\sigma \in G/H} v_{\rho^{-1}\sigma}\,,
\end{eqnarray*}
is a $L/k$-Galois complex. 
\end{example}
Let us denote by $\cC(L)^{\mathrm{Gal}}$ the category of $L/k$-Galois complexes and  $G$-equivariant morphisms. By construction we have a forgetful functor $\cC(L)^{\mathrm{Gal}} \to \cC(L)$. Moreover, the category $\cC(L)^{\mathrm{Gal}}$ carries a canonical symmetric monoidal structure making the forgetful functor symmetric monoidal. The $\otimes$-unit is the complex $L$ (concentrated in degree zero) equipped with the canonical skew-linear left $G$-action, and given $L/k$-Galois complexes $W$ and $W'$ the group $G$ acts diagonally on the underlying tensor product $W\otimes W'$.
\begin{lemma}\label{lem:lifting}
We have the following $\otimes$-functor
\begin{eqnarray}\label{eq:composed}
\cC(l) \too \cC(L)^{\mathrm{Gal}} && V \mapsto \otimes_{\sigma \in G/H} {}^{\sigma\!}V_L\,,
\end{eqnarray}
where $\otimes_{\sigma \in G/H} {}^{\sigma\!}V_L$ is equipped with the skew-linear left $G$-action of Example \ref{ex:skew}.
\end{lemma}
\begin{proof}
Clearly, every morphism $V \to V'$ in $\cC(l)$ gives rise to a $G$-equivariant morphism $\otimes_{\sigma \in G/H} {}^{\sigma\!} V_L \to \otimes_{\sigma \in G/H}{}^{\sigma\!} V_L'$. Hence,  \eqref{eq:composed} is well-defined. The naturality of the following $G$-equivariant isomorphisms 
\begin{eqnarray*}
L \stackrel{\sim}{\to} \otimes_{\sigma \in G/H}{}^{\sigma\!}L &&
(\otimes_{\sigma \in G/H}{}^{\sigma\!} V_L)\otimes(\otimes_{\sigma \in G/H}{}^{\sigma\!} V_L') \stackrel{\sim}{\to} \otimes_{\sigma \in G/H} {}^{\sigma\!}(V \otimes V')_L
\end{eqnarray*}
implies that the functor \eqref{eq:composed} is moreover symmetric monoidal.
\end{proof}
\begin{proposition}{(DG Galois descent)}\label{prop:Galois-descent}
We have a $\otimes$-equivalence of categories
\begin{equation}\label{eq:equivalence-cat}
\xymatrix{
\cC(L)^{\mathrm{Gal}} \ar@/^2ex/[d]^{(-)^G}_{\simeq\;} \\ 
\cC(k) \ar@/^2ex/[u]^{-\otimes_kL} \,,
}
\end{equation}
where $(-)^G$ stands for the $G$-invariants functor.
\end{proposition}
\begin{proof}
If in the above Definition~\ref{def:Galois} we replace the word ``complex'' by the word ``module'', we recover the classical notion of $L/k$-Galois module $W$; see \cite[II.\S5]{Knus}. It consists of a $L$-vector space $W$ equipped with a skew-linear left $G$-action. Let $\mathrm{Vect}(L)^{\mathrm{Gal}}$ be the associated category of $L/k$-Galois modules. As proved in Knus-Ojanguren in \cite[II. Thm.~5.3]{Knus}, we have the following $\otimes$-equivalence of categories
\begin{equation}\label{eq:equiv-cat-aux}
\xymatrix{
\mathrm{Vect}(L)^{\mathrm{Gal}} \ar@/^2ex/[d]^{(-)^G}_{\simeq\;} \\ 
\mathrm{Vect}(k) \ar@/^2ex/[u]^{-\otimes_kL} \,.
}
\end{equation}
Now, note that a $L/k$-Galois complex is precisely the same data as a complex of morphisms in $\mathrm{Vect}(L)^{\mathrm{Gal}}$. Moreover, the symmetric monoidal structure on $\cC(L)^{\mathrm{Gal}}$ is induced by the symmetric monoidal structure on $\mathrm{Vect}(L)^{\mathrm{Gal}}$. Since both functors $-\otimes_kL$ and $(-)^G$ in \eqref{eq:equivalence-cat} are defined degreewise, we hence conclude that \eqref{eq:equivalence-cat} can be obtained from \eqref{eq:equiv-cat-aux} by passing to the category of complexes. This implies that 
\eqref{eq:equivalence-cat} is a $\otimes$-equivalence.
\end{proof}
By combining Proposition~\ref{prop:Galois-descent} with Lemma~\ref{lem:lifting} we obtain the $\otimes$-functor
\begin{eqnarray}\label{eq:func-composed}
\cC(l) \too \cC(k) && V \mapsto (\otimes_{\sigma \in G/H}{}^{\sigma\!} V_L)^G\,.
\end{eqnarray}
\begin{lemma}\label{lem:quasi-isomorphisms}
The above functor \eqref{eq:func-composed} preserves quasi-isomorphisms. Moreover, we have a canonical isomorphism
\begin{equation}\label{eq:can-iso}
(\otimes_{\sigma \in G/H}{}^{\sigma\!} V_L)^G \otimes_{(\otimes_{\sigma \in G/H}{}^{\sigma\!} A_L)^G} (\otimes_{\sigma \in G/H}{}^{\sigma\!} V_L')^G \stackrel{\sim}{\too} (\otimes_{\sigma \in G/H}{}^{\sigma\!} (V\otimes_A V')_L)^G
\end{equation}
for every dg $l$-algebra $A$, right $A$-module $V$, and left $A$-module $V'$.
\end{lemma}
\begin{proof}
Since we are working over a field, the classical K{\"u}nneth formula holds. Hence, the above functor \eqref{eq:composed} preserves quasi-isomorphisms. In order to show that the functor $(-)^G:\cC(L)^{\mathrm{Gal}} \to \cC(k)$ also preserves quasi-isomorphisms, it suffices by equivalence \eqref{eq:equivalence-cat} to show that $-\otimes_kL:\cC(k) \to \cC(L)^{\mathrm{Gal}}$ preserves quasi-isomorphisms. This is clearly the case and so our first claim is proved.

Now, note that we have a canonical isomorphism of complexes of $L$-vector spaces
\begin{equation}\label{eq:canonical}
(\otimes_{\sigma \in G/H}{}^{\sigma\!} V_L) \otimes_{(\otimes_{\sigma \in G/H}{}^{\sigma\!} A_L)} (\otimes_{\sigma \in G/H}{}^{\sigma\!} V_L') \stackrel{\sim}{\too} \otimes_{\sigma \in G/H}{}^{\sigma\!} (V\otimes_A V')_L\,.
\end{equation}
The isomorphism \eqref{eq:canonical} is $G$-equivariant with respect to the skew-linear left $G$-action of Example \ref{ex:skew}; $G$ acts diagonally on the left-hand side. Therefore, the above isomorphism \eqref{eq:can-iso} can be obtained by applying the  functor $(-)^G$ to \eqref{eq:canonical}.
\end{proof}
\subsection{Weil restriction of dg algebras}\label{sub:WeilNC}
Since the above functor \eqref{eq:func-composed} is symmetric monoidal, it gives automatically rise to the following $\otimes$-functor
\begin{eqnarray}\label{eq:Weil-NC}
\cR_{l/k}^\NC:\dgalg(l) \too \dgalg(k) && A \mapsto (\otimes_{\sigma \in G/H}{}^{\sigma\!} A_L)^G\,.
\end{eqnarray}
Note that thanks to the above Lemma \ref{lem:quasi-isomorphisms}, \eqref{eq:Weil-NC} preserves quasi-isomorphisms.
\section{Proof of Theorem~\ref{thm:main1}}\label{sec:proof-main1}
We start by proving that the functor \eqref{eq:Weil-2} preserves proper dg algebras. Let $A$ be a proper dg $l$-algebra. Recall that by definition we have $\sum_i \mathrm{dim}_l H^i(A)< \infty$. The equalities $\mathrm{dim}_L H^i({}^{\sigma\!} A_L)=\mathrm{dim}_l H^i({}^{\sigma\!} A)=\mathrm{dim}_l H^i(A)$, combined with the K{\"u}nneth formula, imply that $\sum_i \mathrm{dim}_L H^i(\otimes_{\sigma \in G/H}{}^{\sigma\!} A_L) < \infty$. Therefore, thanks to the above equivalence of categories \eqref{eq:equivalence-cat}, the proof follows from the following equalities
$$ \sum_i \mathrm{dim}_k H^i(\cR^\NC_{l/k}(A)) = \sum_i \mathrm{dim}_k H^i((\otimes_{\sigma \in G/H}{}^{\sigma\!} A_L)^G) = \sum_i \mathrm{dim}_L H^i(\otimes_{\sigma \in G/H}{}^{\sigma\!} A_L)\,.$$

We now prove that \eqref{eq:Weil-2} preserves smooth dg algebras. The above functor \eqref{eq:func-composed} is symmetric monoidal and, thanks to Lemma~\ref{lem:quasi-isomorphisms}, it preserves quasi-isomorphisms. Therefore, given a dg $l$-algebra $A$, it gives rise to the (non-additive) functor:
\begin{eqnarray}\label{eq:functor-compact23}
\cD(A) \too \cD(\cR_{l/k}^\NC(A)) && M \mapsto (\otimes_{\sigma \in G/H}{}^{\sigma\!} M_L)^G\,.
\end{eqnarray}
\begin{proposition}\label{prop:compact1}
The above functor \eqref{eq:functor-compact23} preserves compact objects.
\end{proposition}
\begin{proof}
Note first that the functor $\cD(A) \to \cD(A), M \mapsto M_L$, as well as the equivalences of categories $\cD(A_L) \stackrel{\simeq}{\to} \cD({}^{\sigma\!}A_L), M \mapsto {}^{\sigma\!} M$, preserve compact objets. Since the triangulated categories $\cD({}^{\sigma\!} A_L)$ and $\cD(\otimes_{\sigma \in G/H}{}^{\sigma\!} A_L)$ are generated by ${}^{\sigma\!} A_L$ and $\otimes_{\sigma \in G/H}{}^{\sigma\!} A_L$, respectively, the following multi-triangulated functor
\begin{eqnarray*}
\prod_{\sigma \in G/H}\cD({}^{\sigma\!} A_L) \too \cD(\otimes_{\sigma \in G/H}{}^{\sigma\!} A_L) && \{{}^{\sigma\!} M\}_{\sigma \in G/H} \mapsto \otimes_{\sigma \in G/H}{}^{\sigma\!} M
\end{eqnarray*}
also preserves compact objects. This implies that the following composition 
\begin{eqnarray}\label{eq:compact3}
\cD(A) \too \cD(\otimes_{\sigma \in G/H}{}^{\sigma\!} A_L) && M \mapsto \otimes_{\sigma \in G/H} {}^{\sigma\!} M_L
\end{eqnarray} 
preserves compact objects. Note that Proposition \ref{prop:Galois-descent} gives rise to the isomorphisms
\begin{eqnarray*}
(\otimes_{\sigma \in G/H}{}^{\sigma\!}A_L)^G\otimes_kL \simeq \otimes_{\sigma \in G/H} {}^{\sigma\!} A_L && (\otimes_{\sigma \in G/H} {}^{\sigma\!}M_L)^G\otimes_k L \simeq \otimes_{\sigma \in G/H} {}^{\sigma\!}M_L\,.
\end{eqnarray*}
Thanks to Lemma \ref{lem:conservativity} below (with $A$ replaced by $(\otimes_{\sigma \in G/H}{}^{\sigma\!}A_L)^G)$, the proof follows now from the fact that the above functor \eqref{eq:compact3} preserves compact objects.
\end{proof}
\begin{lemma}\label{lem:conservativity}
Given a dg $k$-algebra $A$, the following base-change functor 
\begin{eqnarray*}
\cD(A) \to \cD(A\otimes_k L) && M \mapsto M\otimes_k L
\end{eqnarray*}
reflects compact objects.
\end{lemma}
\begin{proof}
Assume that $M\otimes_k L \in \cD_c(A\otimes_k L)$. Since the field extension $L/k$ is finite dimensional, we have $M\otimes_k L \in \cD_c(A)$. Moreover, the choice of a splitting $L \simeq k \oplus C$ of $k$-vector spaces, allows us to express $M$ as a direct summand of $M\otimes_k L$. Therefore, since $M\otimes_k L \in \cD_c(A)$, we conclude that $M$ also belongs to $\cD_c(A)$.
\end{proof}
Thanks to Proposition~\ref{prop:compact1}, the above functor \eqref{eq:functor-compact23} restricts to compact objects 
\begin{eqnarray}\label{eq:functor-compact2}
\cD_c(A) \too \cD_c(\cR_{l/k}^\NC(A))&& M \mapsto (\otimes_{\sigma \in G/H}{}^{\sigma\!} M_L)^G\,.
\end{eqnarray}
Assume now that the dg $l$-algebra $A$ is smooth. Recall that by definition the bimodule \eqref{eq:bimodule1} (with $\cA=A$) belongs to $\cD_c(A^\op \otimes A)$. Therefore, the functor \eqref{eq:functor-compact2} (with $A$ replaced by $A^\op \otimes A$), combined with the fact that $\cR_{l/k}^\NC$ is a $\otimes$-functor, allows us to conclude that the bimodule \eqref{eq:bimodule1} (with $\cA=\cR_{l/k}^\NC(A)$) belongs to $\cD_c(\cR_{l/k}^\NC(A)^\op \otimes \cR_{l/k}^\NC(A))$. This shows that the dg $k$-algebra $\cR_{l/k}^\NC(A)$ is smooth.
\section{Grothendieck group of dg algebras}\label{sec:Grothendieck}
Let $A$ be a dg $k$-algebra. By definition, $K_0(A)$ is the Grothendieck group of the triangulated category $\cD_c(A)$. In this section we give a four-step description of $K_0(A)$, which will be used in the proof of Theorem~\ref{thm:new} and Proposition \ref{prop:binomial}.
\begin{itemize}
\item[(i)] Let $\cC_c(A)$ be the full subcategory of $\cC(A)$ consisting of those right $A$-modules which belong to $\cD_c(A)$. Note that $\cC_c(A)$ is stable under direct sums. Let us then write $M_0(A)$ for the associated monoid.
\item[(ii)] Let $K_0^\oplus(A) $ be the group completion of $M_0(A)$. By construction, we have an  homomorphism $\iota:M_0(A) \to K_0^\oplus(A)$.
\item[(iii)] Recall from \cite[Lem.~3.3]{ICM} that $\cC(A)$ carries an exact structure in the sense of Quillen \cite[\S2]{Quillen}. The conflations are the short exact sequences of $A$-modules
\begin{equation}\label{eq:exact}
\xymatrix{
0 \ar[r] & M \ar[r]^-{i} & M'
\ar[r]_-p & M'' \ar@/_2ex/@{-->}[l]_{s} \ar[r] & 0
}
\end{equation}
for which there exists a morphism $s$ of {\em graded} $A$-modules such that $p\circ s=\id_{M''}$. Note that $\cC_c(A)$ is stable with respect to these short exact sequences. We write $K_0^{ex}(A)$ for the quotient $K_0^\oplus(A)/\langle M'-M-M'' \rangle$, where $M, M', M''$ are as in \eqref{eq:exact}. By construction, we have an homomorphism $K_0^\oplus(A) \twoheadrightarrow K_0^{ex}(A)$.
\item[(iv)] Let $\cQ$ be the set of pairs $(M,M'') \in \cC_c(A) \times \cC_c(A)$ for which there exists a zig-zag of quasi-isomorphisms $M \stackrel{\sim}{\to} \cdot \stackrel{\sim}{\leftarrow} \cdot \cdots \cdot \stackrel{\sim}{\to} \cdot \stackrel{\sim}{\leftarrow} M''$ relating them.
\end{itemize}
\begin{lemma}\label{lem:quotient-group}
The Grothendieck group $K_0(A)$ is isomorphic to the quotient
\begin{equation*}
K_0^{ex}(A) /\langle M-M''\,|\, (M,M'')\in \cQ \rangle \,.
\end{equation*}
\end{lemma}
\begin{proof}
Note first that since $K_0(A)$ is a group, the above homomorphism $\iota$ factors through $K_0^\oplus(A)$. By construction, every conflation \eqref{eq:exact} becomes a distinguished triangle in $\cD_c(A)$. Hence, the homomorphism $\iota$ factors moreover through $K_0^{ex}(A)$. Clearly, quasi-isomorphic right $A$-modules give rise to the same element of $K_0(A)$. This implies that $\iota$ descends furthermore to a surjective group homomorphism
\begin{equation}\label{eq:descends}
\iota:K_0^{ex}(A)/\langle M- M'' \,|\, (M,M'') \in \cQ \rangle \twoheadrightarrow K_0(A)\,.
\end{equation}
Let us now show that \eqref{eq:descends} moreover injective. Note that two right $A$-modules $M, M'' \in \cC_c(A)$ become isomorphic in $\cD_c(A)$ if and only if there exists a zig-zag of quasi-isomorphisms $M \stackrel{\sim}{\to} \cdot \stackrel{\sim}{\leftarrow} \cdot \cdots \cdot \stackrel{\sim}{\to} \cdot \stackrel{\sim}{\leftarrow} M''$ relating them. Since every triangle in $\cD_c(A)$ can be represented (up to quasi-isomorphism) by a conflation in $\cC_c(A)$, we hence conclude that the homomorphism  \eqref{eq:descends} is injective.
\end{proof}
\section{Polynomial maps}
In this short section we recall from Eilenberg-MacLane \cite{EM} the notion of a polynomial map, which will be heavily used in the sequel. 

Let $f:M \to S$ be a map from an abelian monoid to an abelian group such that $f(0)=0$. Following Eilenberg-MacLane, we declare $\Delta^0f:=f$ and define the maps $\Delta^{k+1}f: M^{k+1} \to S, k \geq 0$, by the following recursive formula:
\begin{eqnarray*}
\Delta^{k+1}f(m_0, \ldots, m_{k-1},m_k,m_{k+1}) & := &  \Delta^kf(m_0, \ldots, m_{k-1}, m_k + m_{k+1})\\
&- &  \Delta^kf(m_0, \ldots, m_{k-1}, m_k) \\
&- &  \Delta^kf(m_0, \ldots, m_{k-1}, m_{k+1})\,.
\end{eqnarray*}
\begin{definition}
A map $f:M \to S$ is called {\em polynomial} if there exists an integer $N \geq 0$ such that $\Delta^Nf=0$. The smallest such $N$ is called the degree of $f$.
\end{definition}
\begin{proposition}{(see Joukhovitski \cite[Props. 1.2, 1.6 and 1.7]{Seva})}\label{prop:properties}
\begin{itemize}
\item[(i)] Every polynomial map $f:M \to S$ extends uniquely to a polynomial map $\overline{f}:M^\oplus \to S$ defined on the group completion of $M$.
\item[(ii)] Let $f$ be a polynomial map as in (i) and $\Omega \subseteq M \times M$. Assume that $f(m+m')=f(m+m'')$ for every $m \in M$ and $(m',m'') \in \Omega$. Under these assumptions, $\overline{f}$ factors through the quotient group $M^\oplus/\langle m'-m''\,|\, (m',m'') \in \Omega \rangle$.
\item[(iii)] Given polynomial maps $f : S \to S'$ and $f': S' \to S''$ between abelian groups, their composition $f' \circ f:S \to S''$ is also polynomial. 
\item[(iv)] Every polylinear map $M_1 \times \cdots \times M_n \to S$ is polynomial.
\end{itemize}
\end{proposition}
\section{Binomial rings}\label{sec:binomial}
Recall from Xantcha \cite[\S1]{Xantcha} that a {\em binomial ring} is an unital commutative ring $R$ equipped with unitary operations $r \mapsto \binom{r}{n}, n \in \bbN$, subject to the following axioms:
\begin{itemize}
\item[(i)] $\binom{a+b}{n} = \sum_{p+q=n} \binom{a}{p}\binom{b}{q}$.
\item[(ii)] $\binom{ab}{n} = \sum_{m=0}^n \binom{a}{m} \sum_{q_1 + \cdots + q_m=n, q_i\geq 1} \binom{b}{q_1}\cdots \binom{b}{q_m}$.
\item[(iii)] $\binom{a}{m}\binom{a}{n} = \sum_{k=0}^n \binom{a}{m+k} \binom{m+k}{n}\binom{n}{k}$.
\item[(iv)] $\binom{1}{n}=0$ when $n \geq 2$.
\item[(v)] $\binom{a}{0}=1$ and $\binom{a}{1}=a$.
\end{itemize}
Equivalently, $R$ is a torsion-free commutative ring which is closed in $R_\bbQ$ under the operations $r \mapsto \frac{r(r-1) \cdots (r-n+1)}{n!}$. As explained in {\em loc. cit.}, examples include $\bbZ$, $\bbZ[1/r]$, $\bbQ$, the $p$-adic numbers $\bbZ_p$, and also every $\bbQ$-algebra. 
\subsection*{Compatibility with polynomial maps}
Let $f:S \to S'$ be a polynomial map between abelian groups. As proved in \cite[Thm.~10]{Xantcha}, $f$ extends to a polynomial map $f_R: S_R \to S'_R$. The following result will be used in the proof of Proposition~\ref{prop:binomial}.
\begin{lemma}\label{lem:binomial}
\begin{itemize}
\item[(i)] Given polynomial maps $f: S \to S'$ and $f':S' \to S''$ between abelian groups, we have $(f' \circ f)_R= f'_R \circ f_R$; 
\item[(ii)] Given polynomial maps $f: S \to S'$ and $h:S'' \to S'''$ between abelian groups, we have $(f \times h)_R = f_R \times h_R$;
\item[(iii)] Given a bilinear map $(g,g'):S \times S' \to S''$ between abelian groups, we have $(g,g')_R=(g_R,g'_R)$.
\end{itemize}
\end{lemma}
\begin{proof}
Recall from \cite[Thm.~10]{Xantcha} that a map $f:S \to S'$ between abelian groups is polynomial (=numerical in the sense of \cite[Def.~5]{Xantcha}) if and only if it extends {\em uniquely} to a natural transformation $f\otimes_\bbZ-: S\otimes_\bbZ - \Rightarrow S'\otimes_\bbZ -$ between functors defined on the category of binomial rings. As a consequence, by composing $f \otimes_\bbZ-$ with $f'\otimes_\bbZ -$, we conclude that $f' \circ f$ is polynomial and that $(f'\circ f)_R=f'_R \circ f_R$. This proves item (i). By taking the product of $f \otimes_\bbZ-$ with $h \otimes_\bbZ -$, we conclude that $f \times h$ is polynomial and that $(f \times h)_R = f_R \times h_R$. This proves item (ii). In what concerns item (iii), the induced natural transformation $(g\times g')\otimes_\bbZ-:(S\times S')\otimes_\bbZ- \Rightarrow S''\otimes_\bbZ-$ allow us to conclude that $(g,g')_R=(g_R,g'_R)$. 
\end{proof}

\section{Proof of Theorem \ref{thm:new} and Proposition \ref{prop:binomial}}
Let $A$ be a dg $l$-algebra. Recall from \eqref{eq:functor-compact2} the construction of the functor
\begin{eqnarray}\label{eq:functor-compact}
\cD_c(A) \too \cD_c(\cR^\NC_{l/k}(A)) && M \mapsto (\otimes_{\sigma \in G/H}{}^{\sigma\!} M_L)^G\,.
\end{eqnarray}
\begin{proposition}\label{prop:polynomial}
The above (non-additive) functor \eqref{eq:functor-compact} gives rise to a polynomial map $K_0(A)\to K_0(\cR^\NC_{l/k}(A))$ of degree $d$.
\end{proposition}
\begin{proof}
In order to simplify the exposition, we will restrict ourselves to the case where the field extension $l/k$ is Galois. The proof of the general case is similar.

Recall from Proposition~\ref{prop:Galois-descent} that we have a $\otimes$-equivalence of categories $(-)^G: \cC(L)^\mathrm{Gal} \stackrel{\simeq}{\to} \cC(k)$. Since $\cR^\NC_{l/k}(A):=(\otimes_{\sigma \in G}{}^{\sigma\!} A)^G$, we hence obtain an induced $\otimes$-equivalence of categories $\cC_c(\otimes_{\sigma \in G}{}^{\sigma\!} A)^{\mathrm{Gal}} \simeq \cC_c(\cR^\NC_{l/k}(A))$ and induced isomorphisms:
\begin{eqnarray*}
M_0(\otimes_{\sigma \in G}{}^{\sigma\!} A)^{\mathrm{Gal}}\simeq M_0(\cR_{l/k}^\NC(A)) &&
 K_0^\oplus(\otimes_{\sigma \in G}{}^{\sigma\!} A)^{\mathrm{Gal}} \simeq K_0^\oplus(\cR_{l/k}^\NC(A))\\
K_0^{ex}(\otimes_{\sigma \in G}{}^{\sigma\!} A)^{\mathrm{Gal}} \simeq  K_0^{ex}(\cR_{l/k}^\NC(A)) &&
 K_0(\otimes_{\sigma \in G}{}^{\sigma\!} A)^{\mathrm{Gal}} \simeq K_0(\cR_{l/k}^\NC(A))\,.
\end{eqnarray*}
The upperscripts $(-)^{\mathrm{Gal}}$ emphasize that these constructions are obtained from $\cC(l)^{\mathrm{Gal}}$ and {\em not} from $\cC(l)$. Under these notations, it is enough to show that
\begin{eqnarray}\label{eq:f1}
\cC_c(A) \too \cC_c(\otimes_{\sigma \in G}{}^{\sigma\!} A)^{\mathrm{Gal}} && M \mapsto \otimes_{\sigma \in G} {}^{\sigma\!} M
\end{eqnarray}
gives rise to a polynomial map $K_0(A) \to K_0(\otimes_{\sigma \in G}{}^{\sigma\!} A)^{\mathrm{Gal}}$. The above functor \eqref{eq:f1} is the standard example of a polynomial functor of degree $d$ between additive categories; see \cite[Def.~1.3]{Seva}. Hence, we conclude from \cite[Prop.~1.8]{Seva} that
\begin{eqnarray}\label{eq:induced1}
K_0^\oplus(A) \too K_0^\oplus(\otimes_{\sigma \in G}{}^{\sigma\!} A)^{\mathrm{Gal}} && M\mapsto \otimes_{\sigma \in G}{}^{\sigma\!} M
\end{eqnarray}
is a polynomial map of degree $d$. We now show that the following composition
\begin{equation}\label{eq:induced2}
K_0^\oplus(A) \stackrel{\eqref{eq:induced1}}{\too} K_0^\oplus (\otimes_{\sigma \in G}{}^{\sigma\!} A)^{\mathrm{Gal}} \twoheadrightarrow  K_0^{ex} (\otimes_{\sigma \in G}{}^{\sigma\!} A)^{\mathrm{Gal}}
\end{equation}
descends to $K_0^{ex}(A)$. Let $\Omega \subseteq M_0(A) \times M_0(A)$ be the set of pairs $(M',M\oplus M'')$ associated to the conflations \eqref{eq:exact}. Since $[M\oplus M'']=[M]+ [M'']$ in $K_0^\oplus(A)$, the group $K_0^{ex}(A)$ identifies with the quotient $K_0^\oplus(A)/\langle [M']-[M \oplus M'']\rangle$. Thanks to Proposition~\ref{prop:properties}(ii), it suffices then to show that for every $N \in \cC_c(A)$ the equality 
\begin{equation}\label{eq:searched}
[\otimes_{\sigma \in G}{}^{\sigma\!} (N \oplus M')] = [\otimes_{\sigma \in G}{}^{\sigma\!} (N\oplus M \oplus M'')]
\end{equation}
holds in $K_0^{ex}(\otimes_{\sigma \in G}{}^{\sigma\!} A)^{\mathrm{Gal}}$. Consider the following objects in $\cC_c(\otimes_{\sigma \in G}{}^{\sigma\!} A)^{\mathrm{Gal}}$:
\begin{eqnarray*}\label{eq:graded}
\mathrm{Gr}_i(\otimes_{\sigma \in G}{}^{\sigma\!}(N\oplus M \oplus M'')) := \bigoplus_{\underset{\# I =i}{I \subseteq G}} \otimes_{\sigma \in G} \left\{
  \begin{array}{lcr}{}^{\sigma\!} (N\oplus M)  & \mathrm{if} & \sigma \in I\\
    {}^{\sigma\!} (N \oplus M'') & \mathrm{if} & \sigma \notin I \\
  \end{array}
\right. & 0 \leq i \leq d\,.&
\end{eqnarray*}
Note that the tensor product $\otimes_{\sigma \in G}{}^{\sigma\!} (N\oplus M \oplus M'')$ is canonically isomorphic to the direct sum $\oplus_{i=0}^n \mathrm{Gr}_i(\otimes_{\sigma \in G}{}^{\sigma\!}(N\oplus M \oplus M''))$. As a consequence, we obtain
\begin{equation}\label{eq:equality1}
[\otimes_{\sigma \in G}{}^{\sigma\!} (N\oplus M \oplus M'')] = \sum_{i=0}^d [\mathrm{Gr}_i(\otimes_{\sigma \in G}{}^{\sigma\!}(N\oplus M \oplus M''))]\,.
\end{equation}
Note also that \eqref{eq:exact} gives rise to the following conflation
\begin{equation}\label{eq:conflation}
\xymatrix{
0 \ar[r] & N\oplus M \ar[r]^-{\id \oplus i} & N \oplus M'
\ar[r]_-{\id \oplus p} & N\oplus M'' \ar@/_2ex/@{-->}[l]_{\id \oplus s} \ar[r] & 0\,.
}
\end{equation}
Consider then the following objects:
\begin{eqnarray*}
\mathrm{F}_i(\otimes_{\sigma \in G}{}^{\sigma\!} (N\oplus M')) := \sum_{\underset{\# I =i}{I \subseteq G}} \otimes_{\sigma \in G} \left\{
  \begin{array}{lcr}{}^{\sigma\!} (N \oplus M)  & \mathrm{if} & \sigma \in I\\
    {}^{\sigma\!} (N\oplus M') & \mathrm{if} & \sigma \notin I \\
  \end{array}
\right. & 0 \leq i \leq d\,. &
\end{eqnarray*}
They form a decreasing filtration of  $\otimes_{\sigma \in G}{}^{\sigma\!} (N \oplus M)$ and give rise to the conflations
$$
\xymatrix@C=2em@R=1.5em{
0 \ar[d] &\\
\mathrm{F}_i(\otimes_{\sigma \in G}{}^{\sigma\!} (N \oplus M')) \ar[d] &\\
 \mathrm{F}_{i-1}(\otimes_{\sigma \in G}{}^{\sigma\!} (M\oplus M')) \ar[d] & 1\leq i \leq d\,,\\
 \mathrm{Gr}_{i-1}(\otimes_{\sigma \in G}{}^{\sigma\!} (N\oplus M \oplus M'')) \ar@/_3ex/@{-->}[u] \ar[d]&\\
 0 &}
$$
where the splitting is induced from the one of \eqref{eq:conflation}. As a consequence, we obtain\begin{equation}\label{eq:equality2}
[\otimes_{\sigma \in G}{}^{\sigma\!} (N \oplus M')] = [\mathrm{F}_0(\otimes_{\sigma \in G}{}^{\sigma\!} (N \oplus M')]+ \sum_{i=1}^d [\mathrm{Gr}_i(\otimes_{\sigma \in G}{}^{\sigma\!}(N\oplus M \oplus M''))]\,.
\end{equation}
Finally, since $\mathrm{F}_0(\otimes_{\sigma \in G}{}^{\sigma\!} (N \oplus M')) = \mathrm{Gr}_0(\otimes_{\sigma \in G}{}^{\sigma\!}(N\oplus M \oplus M''))$, the searched equality \eqref{eq:searched} follows from the combination of \eqref{eq:equality1} with \eqref{eq:equality2}. 

We now show that the following composition
$$ K_0^{ex}(A) \stackrel{\eqref{eq:induced2}}{\too} K_0^{ex}(\otimes_{\sigma \in G}{}^{\sigma\!} A)^{\mathrm{Gal}} \twoheadrightarrow K_0(\otimes_{\sigma \in G}{}^{\sigma\!} A)^{\mathrm{Gal}}$$
descends furthermore to $K_0(A)$. Let $\cQ$ be the set of pairs $(M,M'')$ described at item (iv) of \S\ref{sec:Grothendieck}. Thanks to Lemma~\ref{lem:quotient-group}, $K_0(A)$ identifies with the quotient  $K_0^{ex}(A)/\langle M- M''\,|\, (M,M'') \in \cQ \rangle$. Thanks once again to Proposition~\ref{prop:properties}(ii), it suffices to show that for every $N \in \cC_c(A)$ the following equality 
\begin{equation}\label{eq:equality}
[\otimes_{\sigma \in G}{}^{\sigma\!}(N\oplus M)] = [\otimes_{\sigma \in G}(N\oplus M'')]
\end{equation}
holds in $K_0(\otimes_{\sigma \in G}{}^{\sigma\!} A)^{\mathrm{Gal}}$. The functor \eqref{eq:f1} preserves quasi-isomorphisms. Therefore, since by hypothesis $M$ and $M'$ are related by a zig-zag of quasi-isomorphisms, we conclude that $N \oplus M$ and $N \oplus M''$ are also related by a zig-zag of quasi-isomorphisms. This implies the above equality \eqref{eq:equality}.
\end{proof}
Let $A, A' \in \spdgalg(l)$. By applying the above Proposition~\ref{prop:polynomial} to $A^\op \otimes A'$, and using the fact that $\cR_{l/k}^\NC$ is a $\otimes$-functor, we obtain a polynomial map of degree $d$
\begin{equation}\label{eq:poly2}
K_0(A^\op \otimes A') \too K_0(\cR_{l/k}^\NC(A)^\op \otimes \cR_{l/k}^\NC(A'))\,.
\end{equation}
We now have all the ingredients necessary for the definition of the $\otimes$-functor \eqref{eq:we}. By construction of Kontsevich's category of noncommutative Chow motives, it suffices to treat the case of smooth proper dg $l$-algebras. Recall from \eqref{eq:Hom-smooth} that 
\begin{eqnarray*}
\Hom_{\NChow(l)}(U(A),U(A')) &\simeq& K_0(A^\op \otimes A')\\
\Hom_{\NChow(k)}(U(\cR_{l/k}^\NC(A)),U(\cR_{l/k}^\NC(A'))) &\simeq& K_0(\cR_{l/k}^\NC(A)^\op \otimes \cR_{l/k}^\NC(A'))\,.
\end{eqnarray*}
The searched functor \eqref{eq:we} is defined by $A \mapsto \cR^\NC_{l/k}(A)$ on objects and by the above polynomial maps \eqref{eq:poly2} on morphisms. Since \eqref{eq:Weil-2} is a $\otimes$-functor, this construction is symmetric monoidal and preserves the identity morphisms. It remains then only to prove that the composition law is also preserved. Given $A,A',A'' \in \spdgalg(l)$, we need to show that the following diagram commutes (we wrote $\cR$ instead of $\cR_{l/k}^\NC$ in order to simplify the exposition):
\begin{equation*}
\xymatrix@C=5em@R=2em{
K_0(A^\op \otimes A') \times K_0(A'^\op \otimes A'') \ar[d]_-{\eqref{eq:poly2}\times \eqref{eq:poly2}}\ar[r]^-{-\otimes_{A'}-} & K_0(A^\op \otimes A'') \ar[d]^-{\eqref{eq:poly2}} \\
K_0(\cR(A)^\op \otimes \cR(A')) \times K_0(\cR(A')^\op \otimes \cR(A'')) \ar[r]_-{-\otimes_{\cR(A')}-} & K_0(\cR(A)^\op \otimes \cR(A''))\,.
}
\end{equation*}
Thanks to Proposition \ref{prop:properties}(iii)-(iv), both compositions are polynomial maps. In order to prove that these compositions agree, it suffices by Proposition \ref{prop:properties}(i) to show that their restriction to $M_0(A^\op \otimes A')\times M_0(A'^\op \otimes A'')$ is the same, i.e. that $\cR_{l/k}^\NC(\mathrm{B}) \otimes_{\cR^\NC_{l/k}(A')}\cR^\NC_{l/k}(\mathrm{B}')$ is isomorphic to $\cR^\NC_{l/k}(\mathrm{B} \otimes_{A'}\mathrm{B}')$
for every $\mathrm{B} \in \cD_c(A^\op \otimes A')$ and $\mathrm{B}' \in \cD_c(A'^\op \otimes A'')$. This follows from isomorphism~\eqref{eq:can-iso}.
\subsection*{Proof of Proposition \ref{prop:binomial}}
We prove that the functor \eqref{eq:we} extends to a $\otimes$-functor $\cR^\NC_{l/k}:\NChow_R(l) \to \NChow_R(k)$. The proof that the functors \eqref{eq:Karpenko} also extend to $\otimes$-functors $\cR_{l/k}$ and $\cR^\ast_{l/k}$ is (formally) the same and so we leave it to the reader. 

As mentioned in \S\ref{sec:binomial}, every polynomial map $S \to S'$ between abelian groups extends to a polynomial map $S_R \to S'_R$. Therefore, by tensoring \eqref{eq:poly2} with $R$, we obtain a polynomial map of degree $d$
\begin{equation}\label{eq:polynomial2}
K_0(A^\op \otimes A')_R \too K_0(\cR_{l/k}^\NC(A)^\op \otimes \cR_{l/k}^\NC(A'))_R\,.
\end{equation}
The searched functor $\cR^\NC_{l/k}:\NChow_R(l) \to \NChow_R(k)$ is defined by $A\mapsto \cR^\NC_{l/k}(A)$ on objects and by the above polynomial maps \eqref{eq:polynomial2} on morphisms. Similarly to the proof of Theorem \ref{thm:new}, this construction is symmetric monoidal and preserves the identity morphisms. Given $A, A', A'' \in \spdgalg(l)$, it remains then only to show that the following diagram\footnote{Once again we wrote $\cR$ instead of $\cR^\NC_{l/k}$ in order to simplify the exposition.} commutes:
$$
\xymatrix@C=3em@R2em{
K_0(A^\op \otimes A')_R \times K_0(A'^\op \otimes A'')_R \ar[d]_-{\eqref{eq:polynomial2}\times \eqref{eq:polynomial2}}\ar[r]^-{-\otimes_{A'}-} & K_0(A^\op \otimes A'')_R \ar[d]^-{\eqref{eq:polynomial2}} \\
K_0(\cR(A)^\op \otimes \cR(A'))_R \times K_0(\cR(A')^\op \otimes \cR(A''))_R \ar[r]_-{-\otimes_{\cR(A')}-} & K_0(\cR(A)^\op \otimes \cR(A''))_R\,.
}
$$
Thanks to Lemma~\ref{lem:binomial}, this latter diagram can be obtained by tensoring the preceeding commutative diagram with $R$. This concludes the proof.
\section{Perfect complexes on schemes}\label{sec:schemes}
Given a quasi-projective $k$-scheme $X$, let $\Mod(X)$ be the Grothendieck category of sheaves of $\cO_X$-modules and $\mathrm{Qcoh}(X)$ the full subcategory of quasi-coherent $\cO_X$-modules. We denote by $\cD(X):=\cD(\mathrm{Qcoh}(X))$ the derived category of $X$ and by $\perf(X)$ its full triangulated subcategory of perfect complexes. As proved by Bondal-Van den Bergh in \cite[Thms.~3.1.1 and 3.1.3]{BV}, $\perf(X)$ identifies with the category of compact objects in $\cD(X)$. 

Recall from \cite[\S4.4]{ICM} that the derived dg category $\cD_\dg(\cE)$ of an exact category $\cE$ is defined as the dg quotient $\cC_\dg(\cE)/\cA c_\dg(\cE)$ of the dg category of complexes over $\cE$ by its full dg subcategory of acyclic complexes. Let us denote by $\cD_\dg(X)$ the derived dg category $\cD_\dg(\cE)$, with $\cE:=\Mod(X)$, and by $\perf_\dg(X)$ its full dg subcategory of perfect complexes. Note that $\dgHo(\cD_\dg(X)) \simeq \cD(X)$ and $\dgHo(\perf_\dg(X))\simeq \perf(X)$. 

As proved by Bondal-Van den Bergh in \cite[Thm.~3.1.1 and Cor.~3.1.2]{BV}, the triangulated category of perfect complexes $\perf(X)$ is generated by a single object $\cG$. This naturally motivates the following definition:
\begin{definition}\label{def:dga-scheme}
Let $A_X:= \End_{\perf_\dg(X)}(\cG)$ be the associated dg $k$-algebra.
\end{definition}
The inclusion $A_X \to \perf_\dg(X)$ of dg categories is a Morita equivalence; see \cite[Lem.~3.10]{ICM}. Hence, given any other generator $\cG' \in \perf(X)$, we obtain a zig-zag of Morita equivalences $A_X \to \perf_\dg(X) \leftarrow A_X'$. This shows that the dg $k$-algebra $A_X$ is unique up to Morita equivalence.
\subsection*{Compact generators}
Let $f: X \to X'$ be a morphism of quasi-projective $k$-schemes. The following result(s) will be used in the proof of Theorem \ref{thm:new1}.
\begin{proposition}\label{prop:generator}
Assume the following:
\begin{itemize}
\item[(i)] The canonical morphism $\cO_{X'} \to {\bf R}f_\ast(\cO_X)$ in $\cD(X')$ admits a splitting.
\item[(ii)] The object ${\bf R}f_\ast(\cO_X)$ belongs to the smallest localizing (=closed under arbitrary direct sums) triangulated subcategory of $\cD(X')$ containing $\cO_{X'}$.
\end{itemize}
Given a perfect complex $\cG \in \perf(X')$, we have the following implication:
\begin{equation}\label{eq:implication}
{\bf L}f^\ast(\cG)\in \cD(X) \,\,\mathrm{compact}\,\,\mathrm{generator} \Rightarrow \cG\in \cD(X') \,\,\mathrm{compact}\,\,\mathrm{generator}\,.
\end{equation}
\end{proposition}
\begin{proof}
Let $\cF \in \cD(X')$. We need to show that if $\Hom_{\cD(X')}(\Sigma^n\cG,\cF)=0$ for every $n \in \bbZ$, then $\cF$ is trivial. The classical adjunction 
\begin{equation}\label{eq:adjunction}
\xymatrix{
\cD(X) \ar@/^2ex/[d]^-{{\bf R}f_\ast}\\
\cD(X') \ar@/^2ex/[u]^-{{\bf L}f^\ast}
}
\end{equation}
gives rise to the following isomorphisms 
\begin{equation}\label{eq:iso-aux}
\Hom_{\cD(X')}(\Sigma^n\cG, {\bf R}f_\ast{\bf L}f^\ast(\cF)) \simeq \Hom_{\cD(X)}(\Sigma^n {\bf L} f^\ast(\cG), {\bf L}^\ast(\cF))\,.
\end{equation}
Since $\cG$ is a compact object of $\cD(X')$, condition (ii) implies that (the left-hand side of) \eqref{eq:iso-aux} is equal to zero. Using the left-hand side of \eqref{eq:implication}, we hence conclude that ${\bf L} f^\ast(\cF)$, and consequently ${\bf R}f_\ast {\bf L}f^\ast(\cF)$, is trivial. Now, recall that the projection formula (see \cite[Thm.~2.5.5]{TT}) furnish us the following isomorphism
\begin{equation}\label{eq:projection}
{\bf R}f_\ast {\bf L}f^\ast(\cF) = {\bf R}f_\ast ({\bf L}f^\ast(\cF) \otimes^{{\bf L}}_{\cO_{X}} \cO_{X}) \simeq \cF \otimes^{{\bf L}}_{\cO_{X'}}{\bf R}f_\ast(\cO_X)\,.
\end{equation}
Condition (i), combined with isomorphism \eqref{eq:projection}, implies that $\cF$ is a direct summand of ${\bf R}f_\ast {\bf L}f^\ast(\cF)$. Since ${\bf R}f_\ast {\bf L}f^\ast(\cF)$ is trivial, we hence conclude finally that $\cF$ is also~trivial. 
\end{proof}
\begin{lemma}
Assume that the above morphism $f$ is finite. Under this assumption, the converse of the above implication \eqref{eq:implication} holds.
\end{lemma}
\begin{proof}
Let $\cF \in \cD(X)$. Since the functor ${\bf L} f^\ast$ preserves compact objects, it suffices to show that if $\Hom_{\cD(X)}(\Sigma^n{\bf L}f^\ast(\cG), \cF)=0$ for every $n \in \bbZ$, then $\cF$ is trivial. The adjunction \eqref{eq:adjunction} gives rise to the isomorphisms
$$ \Hom_{\cD(X)}(\Sigma^n {\bf L}f^\ast(\cG),\cF)\simeq \Hom_{\cD(X')}(\Sigma^n \cG, {\bf R}f_\ast(\cF))\,.$$
Using the right-hand side of \eqref{eq:implication}, we hence conclude that ${\bf R}f_\ast(\cF)$ is trivial. The proof follows now from the fact that, since by assumption the morphism $f$ is finite, the functor ${\bf R}f_\ast: \cD(X) \to \cD(X')$ is conservative.
\end{proof}
\begin{example}\label{ex:generator}
Given a quasi-projective $k$-scheme $X$, the base-change morphism $X_L \to X$, where $X_L:=X\times_{\mathrm{Spec}(k)}\mathrm{Spec}(L)$, is finite and satisfies conditions (i)-(ii) of Proposition \ref{prop:generator}. Consequently, given a perfect complex $\cG \in \perf(X)$, $\cG_L$ is a compact generator of $\cD(X_L)$ if and only if $\cG$ is a compact generator of $\cD(X)$.
\end{example}
\subsection*{Galois descent}
A {\em $L/k$-Galois scheme} is a $L$-scheme $X$ equipped with a left $G$-action $G \times X \to X, (\rho,x)\mapsto \rho(x)$, which is {\em skew-linear} in the sense that the square
$$
 \xymatrix{
X \ar[d] \ar[rr]^-{\rho(-)} && X \ar[d] \\
\mathrm{Spec}(L) \ar[rr]_-{\mathrm{Spec}(\rho^{-1})} && \mathrm{Spec}(L) 
} $$
commutes for every $\rho \in G$. A {\em $L/k$-Galois module over $X$} consists of a $\cO_X$-module $\cF$ equipped with structure isomorphisms $s_\rho:(\rho(-))_\ast(\cF) \stackrel{\sim}{\to} \cF$ satisfying the following cocycle condition $s_{\tau\rho}=s_\tau \circ (\rho(-))_\ast(s_\rho)$. 

Let $\Mod(X;G)$ be the Grothendieck category of $L/k$-Galois modules and $G$-equivariant morphisms, and $\mathrm{Qcoh}(X;G)$ the full subcategory of those $L/k$-Galois modules that are quasi-coherent as $\cO_X$-modules. In what follows, we denote by $\cD(X;G):=\cD(\mathrm{Qcoh}(X;G))$ the derived category of $L/k$-Galois modules and by $\perf(X;G)$ its full triangulated subcategory consisting of those complexes of $L/k$-Galois modules that are perfect as complexes of $\cO_X$-modules. Let $\cD_\dg(X;G)$ be the derived dg category $\cD_\dg(\cE)$, with $\cE:=\Mod(X;G)$, and $\perf_\dg(X;G)$ its full dg subcategory of those complexes of $L/k$-Galois modules that belong to $\perf(X;G)$. Note that $\dgHo(\cD_\dg(X;G)) \simeq \cD(X;G)$ and $\dgHo(\perf_\dg(X;G))\simeq \perf(X;G)$.

As explained in \cite[Prop.~3.4(i)]{Seva}, the geometric quotient $X/G$ exists and is quasi-projective. Let us write $p:X \to X/G$ for the quotient morphism. The following result will be used in the proof of Theorem \ref{thm:new}.
\begin{proposition}\label{prop:quotient}
We have an equivalence of categories (resp. of dg categories):
\begin{eqnarray}\label{eq:equivalences}
\xymatrix{
\mathrm{perf}(X;G) \ar@/^2ex/[d]^{{\bf R}p_\ast(-)^G}_{\simeq\;} \\ 
\mathrm{perf}(X/G)  \ar@/^2ex/[u]^{{\bf L}p^\ast} 
}&&
\xymatrix{
\perf_\dg(X;G) \ar@/^2ex/[d]^{{\bf R}p_\ast(-)^G}_{\simeq\;} \\ 
\perf_\dg(X/G)  \ar@/^2ex/[u]^{{\bf L}p^\ast}\,.
}
\end{eqnarray}
\end{proposition}
\begin{proof}
Let $\mathrm{Vect}(X/G)$ be the category of vector bundles over $X/G$ and $\mathrm{Vect}(X;G)$ the full subcategory of $\Mod(X;G)$ consisting of those $L/k$-Galois modules that are vector bundles as $\cO_X$-modules. As explained in \cite[page 12]{Seva}, the quotient morphism $p$ gives rise to the following equivalence of categories\footnote{Note that \eqref{eq:vector} generalizes the equivalence of categories \eqref{eq:equiv-cat-aux}.}:
\begin{equation}\label{eq:vector}
\xymatrix{
\mathrm{Vect}(X;G) \ar@/^2ex/[d]^{p_\ast(-)^G}_{\simeq\;} \\ 
\mathrm{Vect}(X/G)  \ar@/^2ex/[u]^{p^\ast} \,.
}
\end{equation}
Recall from Thomason-Trobaugh \cite[Cor.~3.9]{TT} that since every quasi-projective scheme admits an ample family of line bundles, we have the following equivalences 
\begin{eqnarray*}
\cD^b(\mathrm{Vect}(X)) \simeq \perf(X) && \cD^b(\mathrm{Vect}(X/G)) \simeq \perf(X/G)\,.
\end{eqnarray*}
Therefore, by applying $\cD^b(-)$ (resp. $\cD^b_\dg(-)$) to \eqref{eq:vector} we obtain \eqref{eq:equivalences}. 
\end{proof}
\section{Proof of Theorem~\ref{thm:new1}}
We start by recalling from Weil \cite{Weil} the construction of $\cR_{l/k}(X)$. Recall that $X_L:= X \times_{\mathrm{Spec}(l)}\mathrm{Spec}(L)$. Every element $\sigma \in G/H$ gives rise to a new $L$-scheme ${}^{\sigma\!} X_L$ (the {\em $\sigma$-conjugate\footnote{Note that the $\sigma$-conjugate ${}^{\sigma\!} X_L$ depends only (up to isomorphism) on the left coset $\sigma H$.} of $X_L$}) which is obtained from $X_L$ by base-change along the isomorphism $\mathrm{Spec}(\sigma^{-1}):\mathrm{Spec}(L) \stackrel{\sim}{\to} \mathrm{Spec}(L)$. The product $\prod_{\sigma \in G/H}{}^{\sigma\!} X_L$, equipped with the following skew-linear left $G$-action, is a $L/k$-Galois scheme
\begin{eqnarray}\label{eq:G-action-product}
&G \times \prod_{\sigma \in G/H}{}^{\sigma\!} X_L \too \prod_{\sigma \in G/H}{}^{\sigma\!} X_L & (\rho, \{{}^{\sigma\!} x\}_{\sigma \in G/H}) \mapsto \{{}^{\rho^{-1} \sigma}x\}_{\sigma \in G/H}\,.
\end{eqnarray}
The Weil restriction $\cR_{l/k}(X)$ of $X$ is the geometric quotient $(\prod_{\sigma \in G/H}{}^{\sigma\!} X_L) / G$. 

The proof of the first claim is now divided into the following five steps:
\medbreak\noindent\textbf{Step 1:} Recall from Definition \ref{def:dga-scheme} that $A_X := \End_{\perf_\dg(X)}(\cG)$, where $\cG$ is a generator of $\perf(X)$. Thanks to Example \ref{ex:generator} (with $k$ replaced by $l$), $\cG_L$ is a generator of $\perf(X_L)$. Consequently, we conclude that the dg $L$-algebra $A_{X_L}:=\End_{\perf_\dg(X_L)}(\cG_L)$ is quasi-isomorphic to $\End_{\perf_\dg(X)}(\cG)_L=(A_X)_L$.
\medbreak\noindent\textbf{Step 2:} We have equivalences $\perf_\dg(X_L) \stackrel{\simeq}{\to} \perf_\dg({}^{\sigma\!} X_L), \cF \mapsto {}^{\sigma\!} \cF$, of dg categories. This implies that ${}^{\sigma\!} \cG_L$ is a generator of $\perf({}^{\sigma\!} X_L)$ and consequently that the dg $L$-algebra $A_{{}^{\sigma\!} X_L}$ is quasi-isomorphic to ${}^{\sigma\!} A_{X_L}$. 
\medbreak\noindent\textbf{Step 3:} As explained in \cite[Lem.~3.4.1]{BV}, $\boxtimes_{\sigma \in G/H}{}^{\sigma\!} \cG_L$ is a generator of the triangulated category $\perf(\prod_{\sigma \in G/H}{}^{\sigma\!} X_L)$ and the dg $L$-algebra $A_{\prod_{\sigma \in G/H}{}^{\sigma\!} X_L}$ is quasi-isomorphic to $\otimes_{\sigma \in G/H} A_{{}^{\sigma\!} X_L}$. We hence obtain the following Morita equivalence
\begin{eqnarray*}
\bigotimes_{\sigma \in G/H} \perf_\dg({}^{\sigma\!} X_L) \too \perf_\dg(\prod_{\sigma \in G/H}{}^{\sigma\!} X_L) && \{{}^{\sigma\!} \cF\}_{\sigma \in G/H} \mapsto \boxtimes_{\sigma \in G/H}{}^{\sigma\!} \cF\,.
\end{eqnarray*}
\medbreak\noindent\textbf{Step 4:} The generator $\boxtimes_{\sigma \in G/H}{}^{\sigma\!} \cG_L$ of $\perf(\prod_{\sigma \in G/H}{}^{\sigma\!} X_L)$ is naturally a $L/k$-Galois module over the above $L/k$-Galois scheme \eqref{eq:G-action-product}. Let $p:\prod_{\sigma \in G/H}{}^{\sigma\!} X_L \to \cR_{l/k}(X)$ for the quotient morphism. As explained in \cite[page~12]{Seva}, we have isomorphisms
\begin{eqnarray*}
 \cR_{l/k}(X)_L \simeq \prod_{\sigma \in G/H}{}^{\sigma\!} X_L &&
 ({\bf R}p_\ast(\boxtimes_{\sigma \in G/H}{}^{\sigma\!} \cG_L)^G)_L \simeq  \boxtimes_{\sigma \in G/H} {}^{\sigma\!} \cG_L\,.
 \end{eqnarray*} 
Therefore, Example \ref{ex:generator} (with $X$ and $\cG$ replaced by $\cR_{l/k}(X)$ and ${\bf R}p_\ast(\boxtimes_{\sigma \in G/H}{}^{\sigma\!} \cG_L)^G$) implies that that ${\bf R} p_\ast(\boxtimes_{\sigma \in G/H}{}^{\sigma\!} \cG_L)^G$ is a generator of the triangulated category $\perf(\cR_{l/k}(X))$. Making use of Propositon \ref{prop:quotient} (with $X$ replaced by $\prod_{\sigma \in G/H}{}^{\sigma\!} X_L$), we hence conclude that the dg $k$-algebra $A_{\cR_{l/k}(X)}$ is quasi-isomorphic to 
\begin{equation}\label{eq:iso-new}
\End_{\perf_\dg(\prod_{\sigma \in G/H}{}^{\sigma\!} X_L;G)}(\boxtimes_{\sigma \in G/H}{}^{\sigma\!} \cG_L)\,.
\end{equation}
\medbreak\noindent\textbf{Step 5:} Since the Galois group $G$ acts on $\boxtimes_{\sigma \in G/H}{}^{\sigma\!} \cG_F$ (by permutation of the $\boxtimes$-factors), it acts also by conjugation on the following dg $L$-algebra
\begin{equation}\label{eq:acts}
\End_{\perf_\dg(\prod_{\sigma \in G/H}{}^{\sigma\!} X_L)}(\boxtimes_{\sigma \in G/H}{}^{\sigma\!} \cG_L)\,.
\end{equation} 
As explained in \S\ref{sec:schemes}, \eqref{eq:iso-new} identifies with the $G$-invariant part of \eqref{eq:acts}. Since the following concatenation of quasi-isomorphisms
\begin{eqnarray*}
\otimes_{\sigma \in G/H}{}^{\sigma\!}(A_X)_L\simeq  \otimes_{\sigma \in G/H}{}^{\sigma\!} A_{X_L} &\simeq  & \otimes_{\sigma \in G/H}A_{{}^{\sigma\!} X_L} \\
 & \simeq & A_{\prod_{\sigma \in G/H}{}^{\sigma\!} X_L} \\
 & \simeq &  \End_{\perf_\dg(\prod_{\sigma \in G/H}{}^{\sigma\!} X_L)}(\boxtimes_{\sigma \in G/H}{}^{\sigma\!} \cG_L)
\end{eqnarray*}
is $G$-equivariant and the functor $(-)^G$ preserves quasi-isomorphisms, we conclude that the dg $k$-algebra $\cR_{l/k}^\NC(A_X):= (\otimes_{\sigma \in G/H}{}^{\sigma\!} (A_X)_L)^G$ is quasi-isomorphic to $A_{\cR_{l/k}(X)}$. 

\medbreak

We now prove the second claim of Theorem \ref{thm:new1}. Recall that by assumption $R$ is a $\bbQ$-algebra. Let $X$ be a smooth projective $l$-scheme. Thanks to the isomorphisms 
\begin{eqnarray*}
K_0(A_X)_R\simeq K_0(X)_R && K_0(\cR_{l/k}^\NC(A_X))_R \simeq K_0(A_{\cR_{l/k}(X)})_R\simeq K_0(\cR_{l/k}(X))_R\,,
\end{eqnarray*}
to the fact that the categories $\Chow^\ast_R(-)$ and $\NChow_R(-)$ are rigid symmetric monoidal, and to the construction of the $R$-linear fully-faithful $\otimes$-functor $\Psi$ (see \cite[\S8]{CvsNC}), it is enough to show that the following diagram commutes:
$$
\xymatrix{
K_0(X)_R \ar[d]_-{\cF \mapsto \cR^\NC_{l/k}(\cF)} \ar[rrr]^-{\mathrm{ch}\cdot \sqrt{\mathrm{Td}_X}}_{\sim} &&& \cZ_{\mathrm{rat}}^\ast(X)_R \ar[d]^-{\alpha \mapsto \cR^\ast_{l/k}(\alpha)} \\
K_0(\cR_{l/k}(X))_R \ar[rrr]_{\mathrm{ch}\cdot \sqrt{\mathrm{Td}_{\cR_{l/k}(X)}}}^{\sim} &&& \cZ_{\mathrm{rat}}^\ast(\cR_{l/k}(X))_R\,.
}
$$
Some explanations are in order: $\cZ_{\mathrm{rat}}^\ast(-)_R:=\bigoplus_{n\in \bbZ} \cZ^n_{\mathrm{rat}}(-)_R$ stands for the $R$-algebra of algebraic cycles of arbitrary codimension modulo rational equivalence, $\mathrm{ch}$ stands for the Chern character, and $\sqrt{\mathrm{Td}}$ stands for the square root\footnote{Given an $R$-algebra $Z$ and a power series $\varphi=1 + \sum_{n \geq 1} a_n t^n \in Z\llbracket t\rrbracket$, recall that the square root $\sqrt{\varphi}$ is defined as $\mathrm{exp}(\frac{1}{2} \mathrm{log}(\varphi))$.} of the Todd class. As mentioned above, $\cR_{l/k}(X)_L$ is canonically isomorphic to $\prod_{\sigma \in G/H}{}^{\sigma\!} X_L$. Since the base-change homomorphism
$$ \cZ_{\mathrm{rat}}^\ast(\cR_{l/k}(X))_R \too \cZ_{\mathrm{rat}}^\ast(\cR_{l/k}(X)_L)_R \simeq  \cZ_{\mathrm{rat}}^\ast(\prod_{\sigma \in G/H}{}^{\sigma\!} X_L)_R$$
is injective and the Chern character and square root of the Todd class are compatible with base-change, it suffices then to show that the following diagram commutes:
$$
\xymatrix{
K_0(X)_R \ar[rrr]^-{\mathrm{ch} \cdot \sqrt{\mathrm{Td}_X}}_\sim \ar[d]_-{\cF \mapsto \boxtimes_{\sigma \in G/H} {}^{\sigma\!} \cF_L} &&& \cZ_{\mathrm{rat}}^\ast(X)_R \ar[d]^-{\alpha \mapsto \prod_{\sigma \in G/H} {}^{\sigma\!} \alpha_L} \\
K_0(\prod_{\sigma \in G/H} {}^{\sigma\!} X_L)_R \ar[rrr]_-{\mathrm{ch} \cdot \sqrt{\mathrm{Td}_{\prod_{\sigma \in G/H} {}^{\sigma\!} X_L}}}^\sim &&& \cZ_{\mathrm{rat}}^\ast(\prod_{\sigma \in G/H} {}^{\sigma\!} X_L)_R\,.
}
$$
This follows now from the classical isomorphisms (see Fulton \cite[\S3 and \S15]{Fulton}):
$$\mathrm{ch}(\boxtimes_{\sigma \in G/H} {}^{\sigma\!} \cF_L)\simeq\prod_{\sigma \in G/H} \mathrm{ch}({}^{\sigma\!}\cF_L) \simeq \prod_{\sigma \in G/H} {}^{\sigma\!}\mathrm{ch}(\cF_L)\simeq  \prod_{\sigma \in G/H} {}^{\sigma\!}\mathrm{ch}(\cF)_L$$
$$ \sqrt{\mathrm{Td}_{\prod_{\sigma \in G/H}{}^{\sigma\!} X_L}} \simeq \prod_{\sigma \in G/H}\sqrt{\mathrm{Td}_{{}^{\sigma\!} X_L}} \simeq \prod_{\sigma \in G/H}{}^{\sigma\!} \sqrt{\mathrm{Td}_{X_L}}\simeq \prod_{\sigma \in G/H}{}^{\sigma\!}(\sqrt{\mathrm{Td}_{X}})_L \,.$$
\section{Proof of Theorem~\ref{thm:new22}}\label{sec:thm22}
Recall from \cite[Prop.~7.2]{Artin} the construction of the base-change functor:
\begin{eqnarray}\label{eq:base-change1}
\NChow(k)_R \too \NChow(L)_R && U_R(A) \mapsto U_R(A\otimes_k L)\,.
\end{eqnarray}
Since \eqref{eq:base-change1} is $R$-linear and symmetric monoidal, it preserves the $\otimes$-ideal $\otimes_{\mathrm{nil}}$; see \S\ref{sec:NCmotives}. As proved in \cite[Thm.~7.1 and Prop.~7.4]{Artin}, it preserves also the $\otimes$-ideals $\mathrm{Ker}(HP^\pm)$ and $\cN$. Consequently, it gives rise to the $R$-linear $\otimes$-functors:
\begin{eqnarray}
\NVoev(k)_R \too \NVoev(L)_R && U_R(A) \mapsto U_R(A\otimes_k L) \label{eq:base1} \\
\NHom(k)_R \too \NHom(L)_R && U_R(A) \mapsto U_R(A\otimes_k L) \label{eq:base2} \\
\NNum(k)_R \too \NNum(L)_R && U_R(A) \mapsto U_R(A\otimes_k L)\,. \label{eq:base3}
\end{eqnarray} 
Recall from \S\ref{sec:NCmotives} that $\bbQ \subseteq R$ in \eqref{eq:base1} and that $R$ is a field in \eqref{eq:base2}.
\begin{proposition}\label{prop:base1}
The above functors \eqref{eq:base1}-\eqref{eq:base3} are faithful.
\end{proposition}
\begin{proof}
Given a smooth proper dg $k$-algebra $A$, consider the homomorphism 
\begin{eqnarray}\label{eq:base-change}
K_0(A)_R \too K_0(A_L)_R && [M]_R \mapsto [M\otimes_k L]_R\,.
\end{eqnarray}
Let us denote by $\sim_{\mathrm{nil}}, \sim_{\mathrm{hom}}$, and $\sim_{\mathrm{num}}$, the equivalence relations on the $R$-modules 
\begin{eqnarray*}
\Hom_{\NChow_R(k)}(U_R(k),U_R(A))& \simeq & K_0(A)_R \\
 \Hom_{\NChow_R(L)}(U_R(L),U_R(A\otimes_k L)) &\simeq & K_0(A\otimes_k L)_R
\end{eqnarray*}
induced by the $\otimes$-ideals $\otimes_{\mathrm{nil}}, \mathrm{Ker}(HP^\pm)$, and $\cN$, respectively. Since the category of noncommutative Chow motives is rigid symmetric monoidal, the above functors \eqref{eq:base1}-\eqref{eq:base3} are faithful if and only if the following implications hold:
\begin{eqnarray}
{[M\otimes_k L]}_R \sim_{\mathrm{nil}} 0 & \Rightarrow & {[M]}_R \sim_{\mathrm{nil}} 0 \label{eq:impli-1} \\
{[M\otimes_k L]}_R \sim_{\mathrm{hom}} 0 & \Rightarrow & {[M]}_R \sim_{\mathrm{hom}} 0 \label{eq:impli-2} \\
{[M\otimes_k L]}_R \sim_{\mathrm{num}} 0 & \Rightarrow & {[M]}_R \sim_{\mathrm{num}} 0 \label{eq:impli-3}\,.
\end{eqnarray}
If $[M\otimes_k L]_R \sim_{\mathrm{nil}} 0$, then there exists an integer $n \gg 0$ such that $[(M\otimes_k L)^{\otimes n}]_R=[M^{\otimes n} \otimes_k L]_R=0$. Since the Galois field extension $L/k$ is finite, the restriction functor $\cD_c(A\otimes_k L) \to \cD_c(A)$ gives rise to an homomorphism $\mathrm{res}: K_0(A^{\otimes n}\otimes_k L)_R \to K_0(A^{\otimes n})_R$ such that the following composition 
$$ K_0(A^{\otimes n})_R \stackrel{[M]_R \mapsto [M\otimes_k L]_R}{\too} K_0(A^{\otimes n}\otimes_k L)_R \stackrel{\mathrm{res}}{\too} K_0(A^{\otimes n})_R$$
is equal to multiplication by $\mathrm{deg}(L/k)$. Using the fact that $\bbQ \subseteq R$, we hence conclude that $[M^{\otimes n}]_R=0$, \ie that $[M]_R\sim_{\mathrm{nil}} 0$. This shows implication \eqref{eq:impli-1}. Implication \eqref{eq:impli-2} follows from the commutative diagrams
$$
\xymatrix{
\NChow(L)_R \ar[r]^-{HP^\pm} & \mathrm{Vect}_{\bbZ/2}(L) && \NChow(L)_R \ar[r]^-{HP^\pm} & \mathrm{Vect}_{\bbZ/2}(R) \\
\NChow(k)_R \ar[u]^-{\eqref{eq:base-change1}} \ar[r]_-{HP^\pm} & \mathrm{Vect}_{\bbZ/2}(k) \ar[u]_-{V^\pm \mapsto V^\pm\otimes_k L}&& \NChow(k)_R \ar[u]^-{\eqref{eq:base-change1}}\ar[ur]_-{HP^\pm} & 
}
$$
and from the faithfulness of the functor $\mathrm{Vect}_{\bbZ/2}(k) \to \mathrm{Vect}_{\bbZ/2}(L),V^\pm \mapsto V^\pm\otimes_k L $. In what concerns implication \eqref{eq:impli-3}, consider the following bilinear form 
\begin{eqnarray*}
\chi:K_0(A)_R\times K_0(A)_R \to R &&(M, M')\mapsto \sum_{n \in \bbZ} (-1)^n \mathrm{dim}_k \Hom_{\cD_c(A)}(M,M'[n])\,.
\end{eqnarray*}
As explained in \cite[\S4]{Kontsevich}, $[M]_R \sim_{\mathrm{num}} 0$ if and only if $\chi([M]_R,[M']_R)=0$ for every $M' \in \cD_c(A)$. Similarly, $[M\otimes_k L]_R \sim_{\mathrm{num}} 0$ if and only if $\chi([M\otimes_k L]_R, [M'']_R)=0$ for every $M'' \in \cD_c(A\otimes_k L)$. Therefore, making use of the following equalities
\begin{eqnarray*}
\mathrm{dim}_k \Hom_{\cD_c(A)}(M,M'[n]) = \mathrm{dim}_L \Hom_{\cD_c(A_L)}(M\otimes_k L,M\otimes_k L'[n]) && n \in \bbZ\,
\end{eqnarray*}
we hence obtain the above implication \eqref{eq:impli-3}.
\end{proof}
Let $A \in \spdgalg(l)$. Thanks to Proposition \ref{prop:Galois-descent}, the dg $L$-algebra $\cR^\NC_{l/k}(A)_L$ is canonically isomorphic to $\otimes_{\sigma \in G/H} {}^{\sigma\!} A_L$. As a consequence, the composition
\begin{equation}\label{eq:comp-last} 
\NChow_R(l) \stackrel{\cR_{l/k}^\NC}{\too} \NChow_R(k) \stackrel{\eqref{eq:base-change1}}{\too} \NChow_R(L)
\end{equation}
sends $U_R(A)$ to $U_R(\otimes_{\sigma \in G/H} {}^{\sigma\!} A_L)$. In order to prove that the functor $\cR^\NC_{l/k}: \NChow_R(l) \to \NChow_R(k)$ descends to noncommutative $\otimes$-nilpotent motives, noncommutative homological motives, and noncommutative numerical motives, it is enough by Proposition \ref{prop:base1} to prove that the above composition \eqref{eq:comp-last} descends to noncommutative $\otimes$-nilpotent motives, noncommutative homological motives, and noncommutative numerical motives, respectively. Similarly to the proof of Proposition \ref{prop:base1}, given right $A$-modules $M, M' \in \cD_c(A)$, we need to show the implications:
\begin{eqnarray}
{[M]}_R \sim_{\mathrm{nil}} {[M']}_R & \Rightarrow & [\otimes_{\sigma \in G/H} {}^{\sigma\!} M_L]_R \sim_{\mathrm{nil}} [\otimes_{\sigma \in G/H} {}^{\sigma\!} M_L']_R \label{eq:implication-1}\\
{[M]}_R \sim_{\mathrm{hom}} {[M']}_R & \Rightarrow & [\otimes_{\sigma \in G/H} {}^{\sigma\!} M_L]_R \sim_{\mathrm{hom}} [\otimes_{\sigma \in G/H} {}^{\sigma\!} M_L']_R \label{eq:implication-2} \\ 
{[M]}_R \sim_{\mathrm{num}} {[M']}_R & \Rightarrow & [\otimes_{\sigma \in G/H} {}^{\sigma\!} M_L]_R \sim_{\mathrm{num}} [\otimes_{\sigma \in G/H} {}^{\sigma\!} M_L']_R \,. \label{eq:implication-3}
\end{eqnarray}
Since $\otimes_{\mathrm{nil}}$, $\mathrm{Ker}(HP^\pm), \cN$, are $\otimes$-ideals, the equivalence relations $\sim_{\mathrm{nil}}$, $\sim_{\mathrm{hom}}, \sim_{\mathrm{num}}$, are multiplicative. Therefore, it suffices to show that the~homomorphisms
\begin{eqnarray*}
K_0(A)_R \too K_0({}^{\sigma\!}A_L)_R && [M]_R \mapsto [{}^{\sigma\!}M_L]_R 
\end{eqnarray*}
preserve the equivalence relations $\sim_{\mathrm{nil}}$, $\sim_{\mathrm{hom}}$, and $\sim_{\mathrm{num}}$. This follows from the isomorphisms $({}^{\sigma\!} A_L)^{\otimes n}\simeq {}^{\sigma\!}(A_L^{\otimes n})$, from the commutative diagrams
$$
\xymatrix@C=2.2em@R=2.5em{
\NChow(L)_R  \ar[r]^-{HP^\pm} & \mathrm{Vect}_{\bbZ/2}(L) && \NChow(L)_R \ar[r]^-{HP^\pm} & \mathrm{Vect}_{\bbZ/2}(R)  \\
\NChow(l)_R \ar[r]_-{HP^\pm} \ar[u]^-{U_R(A) \mapsto U_R({}^{\sigma\!}A_L)}  & \mathrm{Vect}_{\bbZ/2}(l) \ar[u]^-{V^\pm \mapsto {}^{\sigma\!} V_L^\pm}&& \NChow(l)_R \ar[u]^-{U_R(A) \mapsto U_R({}^{\sigma\!}A_L)} \ar[r]_-{HP^\pm}& \mathrm{Vect}_{\bbZ/2}(R) \ar[u]^-{V^\pm \mapsto {}^{\sigma\!} V_L^\pm}\,,
}
$$
and from the equalities $\chi([M_L]_R,[M_L']_R)= \chi([{}^{\sigma\!}M_L]_R,[{}^{\sigma\!}M_L']_R)$, respectively. 

Now, let us denote by $\Voev_R(k)$ and $\Voev_R^\ast(k)$ the categories of {\em $\otimes$-nilpotent motives} with $R$-coefficients defined as the idempotent completion of the quotient categories $\Chow_R(k)/\otimes_{\mathrm{nil}}$ and $\Chow_R^\ast(k)/\otimes_{\mathrm{nil}}$. In the same vein, let $\Hom_R(k)$ and $\Hom^\ast_R(k)$ be the categories of {\em homological motives} with $R$-coefficients defined as the idempotent completion of the quotient categories $\Chow_R(k)/\mathrm{Ker}(H^\ast_{dR})$ and $\Chow^\ast_R(k)/\mathrm{Ker}(H^\ast_{dR})$, where $H^\ast_{dR}$ stands for de Rham cohomology. The proof that the functors $\cR_{l/k}:\Chow_R(l) \to \Chow_R(k)$ and $\cR^\ast_{l/k}:\Chow^\ast_R(l) \to \Chow^\ast_R(k)$ also descend to the categories of $\otimes$-nilpotent motives, homological motives, and numerical motives, is (formally) the same and so we leave it for the reader.

As explained in \cite[Prop.~4.1]{Voevodsky} and in the proof of \cite[Thm.~1.5]{Galois}, the functor $\Psi$ descends to an $R$-linear fully-faithful $\otimes$-functor $\Psi_{\mathrm{nil}}: \Voev_R^\ast(k) \to \NVoev_R(k)$ and to an $R$-linear $\otimes$-functor $\Psi_{\mathrm{hom}}: \Hom^\ast_R(k) \to \NHom_R(k)$. Consequently, the commutativity of the following diagrams 
$$
\xymatrix{
& \Voev^\ast_R(l) \ar[d]_-{\cR^\ast_{l/k}} \ar[r]^-{\Psi_{\mathrm{nil}}} & \NVoev_R(l) \ar[d]^-{\cR^\NC_{l/k}} & \\
& \Voev^\ast_R(k) \ar[r]_-{\Psi_{\mathrm{nil}}} & \NVoev_R(k) & \\
 \Hom^\ast_R(l) \ar[d]_-{\cR^\ast_{l/k}} \ar[r]^-{\Psi_{\mathrm{hom}}} & \NHom_R(l) \ar[d]^-{\cR^\NC_{l/k}} & 
 \Num^\ast_R(k) \ar[r]^-{\Psi_{\mathrm{num}}} \ar[d]_-{\cR^\ast_{l/k}} & \NNum_R(k) \ar[d]^-{\cR^\NC_{l/k}} \\
 \Hom^\ast_R(l)  \ar[r]_-{\Psi_{\mathrm{hom}}} & \NHom_R(l)   &
 \Num^\ast_R(k) \ar[r]_-{\Psi_{\mathrm{num}}} & \NNum_R(k) 
}
$$
follows from the commutativity of \eqref{eq:diagram-Chow} and from the fact that the quotient functors
\begin{eqnarray*}
\Chow_R^\ast(l) \too \Voev^\ast(l)_R & \Chow_R^\ast(l) \too \Hom_R^\ast(l) & \Chow_R^\ast(l) \too \Num_R^\ast(l)
\end{eqnarray*}
are full and essentially surjective (up to direct summands).
\section{Proof of Theorem \ref{thm:new3}}
Recall first from \cite[Prop.~2.25]{separable} that the isomorphisms \eqref{eq:index} are given by
\begin{eqnarray}\label{eq:isos}
\mathrm{ind}(A^\op \otimes A') \cdot \bbZ \simeq K_0(A^\op \otimes A') && \mathrm{ind}(A^\op \otimes A') \mapsto [I_{A^\op \otimes A'}]\,,
\end{eqnarray}
where $I_{A^\op \otimes A'}$ stands for the minimal ideal of the central simple $k$-algebra $A^\op \otimes A'$.

Since the category of noncommutative Chow motives is rigid symmetric monoidal and $\cR^\NC_{l/k}$ is a $\otimes$-functor, we can assume without loss of generality that $A=l$. Recall from the Artin-Wedderburn theorem that $A'\simeq M_{n \times n}(D)$ for a unique integer $n \geq 1$ and division $l$-algebra $D$. Using the fact that $M_{n\times n}(D)$ is Morita equivalent to $D$, we can assume moreover that $A'=D$ is a (finite dimensional) division $l$-algebra. In order to prove Theorem \ref{thm:new3}, we need then to show that the map 
\begin{equation}\label{eq:map-1}
\Hom_{\mathrm{CSA}(l)}(U(l),U(D)) \too \Hom_{\mathrm{CSA}(k)}(U(k),U(\cR^\NC_{l/k}(D)))
\end{equation}
identifies, under the above isomorphisms \eqref{eq:isos}, with the polynomial map
\begin{eqnarray}\label{eq:map-2}
\mathrm{ind}(D)\cdot \bbZ \too \mathrm{ind}(\cR^\NC_{l/k}(D))\cdot \bbZ && n \mapsto n^d\,.
\end{eqnarray}
Assume first that $D=l$. In this case, \eqref{eq:map-1} reduces to the polynomial map
\begin{eqnarray}\label{eq:poly-1}
\bbZ \simeq K_0(l) \too K_0(k) \simeq \bbZ && [M] \mapsto [\cR^\NC_{l/k}(M)]\,.
\end{eqnarray}
Thanks to Lemma \ref{lem:computation} below, we have $[\cR^\NC_{l/k}(l^{\oplus n})]=[k^{\oplus (n^d)}]=n^d$ for every $n \geq 1$. Since the unique polynomial extension of the map $\bbN \to \bbZ, n \mapsto n^d$, is still given by $n \mapsto n^d$ (see Proposition \ref{prop:properties}), we conclude that \eqref{eq:poly-1} identifies with $n \mapsto n^d$.

Let us now prove the general case. Note that the above isomorphisms \eqref{eq:isos} (with $A=l$ and $B=D,\cR_{l/k}^\NC(D)$) reduce to 
\begin{eqnarray}
\mathrm{ind}(D) \cdot \bbZ \simeq K_0(D) && \mathrm{ind}(D) \mapsto [D] \label{eq:iso-1}\\
\mathrm{ind}(\cR^\NC_{l/k}(D)) \cdot \bbZ \simeq K_0(\cR^\NC_{l/k}(D)) && \mathrm{ind}(\cR^\NC_{l/k}(D)) \mapsto [I_{\cR_{l/k}^\NC(D)}] \label{eq:iso-2}\,.
\end{eqnarray}
The above map \eqref{eq:map-1} sends $[D]$ to $[\cR^\NC_{l/k}(D)]$. Hence, the following equality 
$$ [\cR^\NC_{l/k}(D)] =\frac{\mathrm{deg}(\cR^\NC_{l/k}(D))}{\mathrm{ind}(\cR^\NC_{l/k}(D))}[I_{\cR_{l/k}^\NC(D)}]\,,$$
combined with the isomorphisms \eqref{eq:iso-1}-\eqref{eq:iso-2}, implies that the above polynomial map \eqref{eq:map-2} sends $\mathrm{ind}(D)$ to $\mathrm{deg}(\cR^\NC_{l/k}(D))$. Since the degree of a central simple algebra is invariant under base-change, Proposition \ref{prop:Galois-descent} gives rise to the equalities
$$ \mathrm{deg}(\cR^\NC_{l/k}(D))=\mathrm{deg}(\cR^\NC_{l/k}(D)_L)= \mathrm{deg}(\otimes_{\sigma \in G/H}{}^{\sigma\!} D_L)= \mathrm{deg}(D_L)^d = \mathrm{ind}(D)^d\,.$$
This allows us to conclude that \eqref{eq:map-1} identifies with the polynomial map \eqref{eq:map-2} in the particular case when $n=\mathrm{ind}(D)$. Since under the above isomorphisms \eqref{eq:isos} the composition law (=$\mathrm{comp.}$) corresponds to multiplication, all the remaining cases follow now from the following commutative diagram:
$$
\xymatrix{
\Hom(U(l),U(l)) \times \Hom(U(l),U(D)) \ar[d]_-{(n \mapsto n^d)\times \eqref{eq:map-1}}\ar[r]^-{\mathrm{comp.}} & \Hom(U(l),U(D)) \ar[d]^-{\eqref{eq:map-1}} \\
\Hom(U(k),U(k)) \times \Hom(U(k),U(\cR^\NC_{l/k}(D))) \ar[r]_-{\mathrm{comp.}} & \Hom(U(k),U(\cR^\NC_{l/k}(D)))\,.
}
$$
This concludes the proof.
\begin{lemma}\label{lem:computation}
We have an isomorphism of $k$-vector spaces $\cR^\NC_{l/k}(l^{\oplus n}) \simeq k^{\oplus(n^d)}$.
\end{lemma}
\begin{proof}
Similarly to \eqref{eq:computation}, we have an isomorphism of $k$-vector spaces $\cR^\NC_{l/k}(l^{\oplus n}) \simeq \bigoplus_{\alpha \in \cO(G/H,n)}L^{\mathrm{stab}(\alpha)}$. Note that the $k$-vector space $L^{\mathrm{stab}(\alpha)}$ is of dimension $\frac{\mathrm{deg}(L/k)}{\# \mathrm{stab}(\alpha)}$, where $\#\mathrm{stab}(\alpha)$ stands for the cardinality of $\mathrm{stab}(\alpha) \subseteq G$. Making use of the equality $\sum_{\alpha \in \cO(G/H,n)}\frac{\mathrm{deg}(L/k)}{\# \mathrm{stab}(\alpha)}=n^d$, we hence conclude that $\cR^\NC_{l/k}(l^{\oplus n}) \simeq k^{\oplus (n^d)}$.
\end{proof}
\subsection*{Base-change}
It is well-known that the assignment $A \mapsto A\otimes_k l$ preserves central simple algebras. Consequently, the functor $-\otimes_k l: \NChow(k) \to \NChow(l)$ (see \cite[Prop.~7.2]{Artin}) restricts to a $\bbZ$-linear $\otimes$-functor $-\otimes_k l : \mathrm{CSA}(k) \to \mathrm{CSA}(l)$.
\begin{proposition}\label{prop:base2}
Given central simple $k$-algebras $A$ and $A'$, the homomorphism
$$ \Hom_{\mathrm{CSA}(k)}(U(A),U(A')) \too \Hom_{\mathrm{CSA}(l)}(U(A\otimes_k l), U(A' \otimes_k l))$$
identifies, under the isomorphisms \eqref{eq:index}, with the inclusion homomorphism
\begin{equation}\label{eq:homo-inclusion}
\mathrm{ind}(A^\op \otimes A')\cdot \bbZ \hookrightarrow \mathrm{ind}((A^\op \otimes A')\otimes_k l)\cdot \bbZ\,.
\end{equation}
\end{proposition}
\begin{proof}
By combining the isomorphisms \eqref{eq:isos} with the fact that the Grothendieck classes $[A^\op \otimes A']$ and $[(A^\op \otimes A')\otimes_kl]$ are equal respectively to $\frac{\mathrm{deg}(A^\op \otimes A')}{\mathrm{ind}(A^\op \otimes A')}[I_{A^\op \otimes A'}]$ and $\frac{\mathrm{deg}((A^\op \otimes A')\otimes_k l)}{\mathrm{ind}((A^\op \otimes A')\otimes_k l)}[I_{(A^\op \otimes A')\otimes_k l}]$,  
we conclude that the above homomorphism \eqref{eq:homo-inclusion} sends $\mathrm{deg}(A^\op \otimes B)$ to $\mathrm{deg}((A^\op \otimes B)\otimes_k l)$. The proof follow now from the fact that the degree of a central simple algebra is invariant under base-change.
\end{proof}

\end{document}